\documentclass[final,3p,times,authoryear]{elsarticle}

\usepackage{amsmath,amsthm,amsfonts,amscd,amssymb,bm,mathrsfs}
\usepackage{enumerate}
\usepackage{graphicx,xcolor}
\usepackage{epstopdf}
\usepackage{subfigure}
\usepackage{multirow}
\usepackage{cases}
\DeclareMathOperator*{\argmin}{arg\,min}

\newcommand{\B}{{\mathbb{B}}}

\newcommand{\R}{{\mathbb{R}}}

\newcommand{\loss}{{\mathcal{L}}}
\newcommand{\regu}{{\mathcal{R}}}

\newcommand{\Q}{{\mathcal{Q}}}
\newcommand{\h}{{\mathcal{H}}}
\newcommand{\T}{{\mathcal{T}}}

\newtheorem{Assumption}{Assumption}
\newtheorem{Lemma}{Lemma}

\newtheorem{Remark}{Remark}
\newtheorem{Theorem}{Theorem}

\newtheorem{Example}{Example}

\begin{document}

\begin{frontmatter}



\title{Sparse recovery via nonconvex regularized $M$-estimators over $\ell_q$-balls}


\author[1,2]{Xin Li}
\author[3]{Dongya Wu}
\author[1]{Chong Li}
\author[4]{Jinhua Wang}
\author[5]{Jen-Chih Yao}

\address[1]{School of Mathematics, Northwest University, Xi’an, 710069, China}
\address[2]{School of Mathematical Sciences, Zhejiang University, Hangzhou 310027, P. R. China}
\address[3]{School of Information Science and Technology, Northwest University, Xi’an, 710069, China}
\address[4]{Department of Mathematics, Zhejiang University of Technology, Hangzhou 310032, P. R. China}
\address[5]{Center for General Education, China Medical University, Taichung 404, Taiwan}


\begin{abstract}

In this paper, we analyse the recovery properties of nonconvex regularized $M$-estimators, under the assumption that the true parameter is of soft sparsity. In the statistical aspect, we establish the recovery bound for any stationary point of the nonconvex regularized $M$-estimator, under restricted strong convexity and some regularity conditions on the loss function and the regularizer, respectively. In the algorithmic aspect, we slightly decompose the objective function and then solve the nonconvex optimization problem via the proximal gradient method, which is proved to achieve a linear convergence rate. In particular, we note that for commonly-used regularizers such as SCAD and MCP, a simpler decomposition is applicable thanks to our assumption on the regularizer, which helps to construct the estimator with better recovery performance. Finally, we demonstrate our theoretical consequences and the advantage of the assumption by several numerical experiments on the corrupted errors-in-variables linear regression model. Simulation results show remarkable consistency with our theory under high-dimensional scaling.

\end{abstract}

%


\end{frontmatter}


\section{Introduction}\label{sec-intro}

High-dimensional data sets have posed both statistical and computational challenges in recent decades \citep{babu2004some, fan2014challenges, wainwright2014structured}. In the ``large $n$, small $m$'' regime, where $n$ refers to the problem dimension and $m$ refers to the sample size, it is well known that obtaining consistent estimators is impossible unless the model is endowed with some additional structures. Consequently, in the statistical aspect, a variety of research have imposed some low-dimensional constraint on the parameter space, such as sparse vectors \citep{bickel2009simultaneous}, low-rank matrices \citep{recht2010guaranteed}, or structured covariance matrices \citep{li2018efficient}. In the computational aspect, a lot of well-known estimators are formulated as solutions to optimization problems comprised of a loss function with a weighted regularizer, where the loss function measures the data fidelity and the regularizer represents the low-dimensional constraint. Estimators with this formulation are usually referred to as regularized $M$-estimators \citep{agarwal2012fast, negahban2012unified}. For instance, in high-dimensional liner regression, the Lasso \citep{tibshirani1996regression} is based upon solving a convex optimization problem, formed by a combination of the least squares loss and the $\ell_1$-norm regularizer. Significant progress has been achieved in studying the recovery bounds of convex $M$-estimators and designing both effective and efficient numerical algorithms for optimization; see \cite{agarwal2012fast, bickel2009simultaneous, negahban2012unified} and references therein.

Though the convex $M$-estimation problems have gained a great success, nonconvex regularized $M$-estimators have recently attracted increasing attention thanks to the better statistical properties they might enjoy. As an example, nonconvex regularizers such as the smoothly clipped absolute deviation penalty (SCAD) \citep{fan2001variable} and minimax concave penalty (MCP) \citep{zhang2010nearly} can eliminate the estimation bias to some extent and achieve more refined statistical rates of convergence, while the convex $\ell_1$-norm regularizer always induces significant estimation bias for parameters with large absolute values \citep{wang2014optimal, zhang2008sparsity}. Meanwhile, the loss function can also be nonconvex in real applications, such as error-in-variables linear regression; see \cite{carroll2006measurement, loh2012high} and references therein.

Standard statistical results for nonconvex $M$-estimators often only provide recovery bound for global solutions \citep{fan2001variable, zhang2010nearly, zhang2012general}, while several numerical methods previously proposed to optimize nonconvex functions, such as local quadratic approximation (LQA) \citep{fan2001variable}, minorization-maximization
(MM) \citep{hunter2005variable}, local linear approximation (LLA) \citep{zou2008one}, and coordinate descent \citep{breheny2011coordinate}, may attain the local solutions. This results in a noticeable gap between theory and practice. Therefore, it is necessary to analyse the statistical properties of the local solutions obtained by certain numerical procedures.

Recently, researchers in \cite{wang2014optimal} and \cite{loh2015regularized} have independently investigated the local solutions of nonconvex regularized $M$-estimators that can be formulated as
\begin{equation}\label{M-esti-intr}
\hat{\beta} \in \argmin_{\beta\in \Omega\subseteq \R^n}\{\loss_m(\beta)+\regu_\lambda(\beta)\},
\end{equation}
where $\loss_m(\cdot)$ is the loss function and $\regu_\lambda(\cdot)$ is the regularizer with regularization parameter $\lambda$. Both of $\loss_m$ and $\regu_\lambda$ can be nonconvex. \cite{wang2014optimal} proposed an approximate regularization path-following method leveraging the proximal gradient method \citep{nesterov2007gradient} within each path-following stage. The recovery bounds for all the approximate local solutions along the full regularization path were established. \cite{loh2015regularized} considered the $M$-estimator \eqref{M-esti-intr} with a convex side constraint $\Omega=\{\beta: g(\beta)\leq r\}$. They proved that any stationary points of the nonconvex optimization problem lie within statistical precision of the true parameter and modified the proximal gradient method to obtain a near-global optimum.

However, these works both focus on the sparsity assumption that the underlying parameter $\beta^*$ is exact sparse, i.e., $\|\beta^*\|_0=s$, which may be too strict for some problems. Let us consider the standard linear regression $y=\sum_{j=1}^n\beta^*_jx_j+e$ with $e$ being the observation noise. The exact sparsity assumption means that only a small subset of entries of the regression coefficients $\beta^*_j$'s are nonzeros or equivalently most of the covariates $x_j$'s absolutely have no effect on the response $y$, which is sometimes too restrictive in real applications. For instance, in image processing, it is standard that wavelet coefficients for images always exhibit an exponential decay, but do not need to be almost $0$ (see, e.g., \cite{joshi1995image, lustig2007sparse}). Other applications in high-dimensional scenarios include signal processing \citep{candes2006stable}, medical imaging reconstruction \citep{lustig2008compressed}, data mining \citep{orre2000bayesian} and so on, where it is not reasonable to impose exact sparsity assumption on the model space. Therefore, it is necessary to investigate the statistical properties of the nonconvex $M$-estimators when the exact sparsity assumption does not hold. In addition, as for the algorithmic aspect, the numerical procedures proposed in \cite{loh2015regularized} is based on the regularity conditions on the regularizer. Particularly, for commonly-used regularizer such as the SCAD and MCP, the side constraint is set as $g(\cdot)=\frac{1}{\lambda}\left\{\regu_\lambda(\cdot)+\frac{\mu}{2}\|\cdot\|_2^2\right\}$ for a suitable constant $\mu>0$, and thus $g(\cdot)$ is a piecewise function. The iteration takes the form
\begin{equation}\label{Loh-ite}
\beta^{t+1}\in \argmin_{\beta\in \R^n, g(\beta)\leq r}\left\{\frac{1}{2}\Big{\|}\beta-\left(\beta^t-\frac{\nabla{\bar{\loss}_m(\beta^t)}}{v}\right)\Big{\|}_2^2+\frac{\lambda}{v}g(\beta)\right\},
\end{equation}
where $\bar{\loss}_m(\cdot)=\loss_m(\cdot)-\frac{\mu}{2}\|\cdot\|_2^2$, and $\frac{1}{v}$ is the step size. This process involves projection onto the sublevel set of a piecewise function, which may cost a large amount of computation in the high-dimensional scenario.

Our main purpose in the present paper is to deal with the more general case that the coefficients of the true parameter are not almost zeros and to design a algorithm with better recovery performance. More precisely, we assume that for $q\in [0,1]$ fixed, the $\ell_q$-norm of $\beta^*$ defined as $\|\beta^*\|_q^q:=\sum_{j=1}^n|\beta^*_j|^q$ is bounded from above by a constant. Note that this assumption is reduced to the exact sparsity assumption aforementioned when $q=0$. When $q\in (0,1]$, this type of sparsity is known as the soft sparsity and has been used to analyse the minimax rate for linear regression \citep{raskutti2011minimax}. In the aspect of computation, we apply the proximal gradient method \citep{nesterov2007gradient} to solve a modified version of the nonconvex optimization problem \eqref{M-esti-intr}. The main contributions of this paper are as follows.  First, under the general sparsity assumption on the true parameter $\beta^*$ (i.e., $\|\beta^*\|_q^q\leq R_q,\ q\in[0,1]$), we provide the $\ell_2$ recovery bound for the stationary point $\tilde{\beta}$ of the nonconvex optimization problem as $\|\tilde{\beta}-\beta^*\|_2^2=O(\lambda^{2-q}R_q)$. When $\lambda$ is chosen as $\lambda=\Omega\left(\sqrt{\frac{\log n}{m}}\right)$, the recovery bound implies that the any stationary point is statistically consistent; see Theorem \ref{thm-sta}. Second, we consider the more general case that the regularizer can be decomposed as $\regu_\lambda(\cdot)=\h_\lambda(\cdot) +\Q_\lambda(\cdot)$, where $\h_\lambda$ is convex and $\Q_\lambda$ is concave. By virtue of this assumption, we establish that the proximal gradient algorithm linearly converge to a global solution of the nonconvex regularized problem; see Theorem \ref{thm-algo}. Since the proposed algorithm relies highly on the decomposition of the regularizer, this more general condition provides us the potential to consider different decompositions for the regularizer so as to construct different numerical iterations. In particular, for the SCAD and MCP regularizer, we can choose $\h_\lambda(\cdot)=\lambda\|\cdot\|_1$, then the iterative sequence is generated as follows
\begin{equation}\label{Mine-ite}
\beta^{t+1}\in \argmin_{\beta\in \R^n, \|\beta\|_1\leq r}\left\{\frac{1}{2}\Big{\|}\beta-\left(\beta^t-\frac{\nabla{\bar{\loss}_m(\beta^t)}}{v}\right)\Big{\|}_2^2+\frac{\lambda}{v}\|\beta\|_1\right\},
\end{equation}
where $\bar{\loss}_m(\cdot)=\loss_m(\cdot)+\Q_\lambda(\cdot)$, and $\frac1v$ is the step size. This numerical procedure involves a soft-threshold operator and $\ell_2$ projection onto the $\ell_1$-ball of radius $r$, and thus requires lower computational cost than \eqref{Loh-ite}. The advantage of iteration \eqref{Mine-ite} is illustrated in Fig. \ref{f-com-add} and \ref{f-com-mis}.

The remainder of this paper is organized as follows. In section \ref{sec-prob}, we provide background on nonconvex $M$-estimation problems and some regularity conditions on the loss function and the regularizer. In section \ref{sec-main}, we establish our main results on statistical consistency and algorithmic rate of convergence. In section \ref{sec-simul}, we perform several numerical experiments to demonstrate our theoretical results. We conclude this paper in Section \ref{sec-concl}. Technical proofs are presented in Appendix.

We end this section by introducing some notations for future reference. We use Greek lowercase letters $\beta,\delta$ to denote the vectors, capital letters $J,S$ to denote the index sets. For a vector $\beta\in \R^n$ and an index set $S\subseteq \{1,2,\dots,n\}$, we use $\beta_S$ to denote the vector in which $(\beta_S)_i=\beta_i$ for $i\in S$ and zero elsewhere, $|S|$ to denote the cardinality of $S$, and $S^c=\{1,2,\dots,n\}\setminus S$ to denote the complement of $S$. A vector $\beta$ is supported on $S$ if and only if $S=\{i\in \{1,2,\dots,n\}:\beta_i\neq 0\}$, and $S$ is the support of $\beta$ denoted by ${\text{supp}}(\beta)$, namely ${\text{supp}}(\beta)=S$. For $m\geq 1$, let $\mathbb{I}_m$ stand for the $m\times m$ identity matrix. For a matrix $X\in \R^{m\times n}$, let $X_{ij}\ (i=1,\dots,m,j=1,2,\cdots,n)$ denote its $ij$-th entry, $X_{i\cdot}\ (i=1,\dots,m)$ denote its $i$-th row, $X_{\cdot j}\ (j=1,2,\cdots,n)$ denote its $j$-th column, and diag$(X)$ stand for the diagonal matrix with its diagonal entries equal to $X_{11},X_{22},\cdots,X_{nn}$. We write $\lambda_{\text{min}}(X)$ and $\lambda_{\text{max}}(X)$ to denote the minimal and maximum eigenvalues of a matrix $X$, respectively. For a function $f:\R^n\to \R$, $\nabla f$ is used to denote a gradient or subgradient depending on whether $f$ is differentiable or nondifferentiable but convex, respectively.

\section{Problem setup}\label{sec-prob}

In this section we begin with a precise formulation
of the problem, and then impose some suitable assumptions on the loss function as well as the regularizer.

\subsection{Nonconvex regularized $M$-estimation}

Following the work of \cite{negahban2012unified, wainwright2014structured}, we first review some basic concepts on the $M$-estimation problem. Let $Z_1^m:=(Z_1,Z_2,\cdots,Z_m)$ denote a sample of $m$ identically independent observations of a given random variable $Z:\mathcal{S} \to \mathcal{Z}$ defined on the probability space $(\mathcal{S},\mathcal{F},\mathbb{P})$, where $\mathbb{P}$ lies within a parameterized set $\mathcal{P}=\{\mathbb{P}_\beta: \beta\in \Theta \subseteq \R^n\}$. It is always assumed that there is a ``true" probability distribution $\mathbb{P}_{\beta^*}\in \mathcal{P}$ that generates the observed data $Z_1^m$ and the goal is to estimate the unknown true parameter $\beta^*\in \Theta$. To this end, a loss function $\loss_m:\R^n\times \mathcal{Z}^m\to \R_+$ is introduced, whose value $\loss_m(\beta;Z_1^m)$ measures the ``fit'' between any parameter $\beta\in \Theta$ and the observed data set $Z_1^m\in \mathcal{Z}^m$ and smaller value means better fit.

However, when the number of observations $m$ is smaller than the ambient dimension $n$, consistent estimators can no longer be obtained. Fortunately, there is empirical evidence showing that the underlying true parameter $\beta^*$ in the high-dimensional space is sparse in a wide range of applications; see, e.g., \cite{joshi1995image, lustig2007sparse}. One popular way to measure the degree of sparsity is to use the $\ell_q$-ball\footnote{Accurately speaking, when $q\in[0,1)$, these sets are not real ``ball''s , as they fail to be convex.}, which is defined as, for $q\in [0,1]$, and a radius $R_q>0$,
\begin{equation}\label{lq-ball}
\B_q(R_q):=\{\beta\in \R^n:||\beta||_q^q=\sum_{j=1}^n|\beta_j|^q\leq R_q\}.
\end{equation}
Note that the $\ell_0$-ball corresponds to the case of exact sparsity, meaning that any vector $\beta\in \B_0(R_0)$ is supported on a set of cardinality at most $R_0$, while the $\ell_q$-ball for $q\in (0,1]$ corresponds to the case of soft sparsity, which enforces a certain decay rate on the ordered elements of $\beta\in \B_q(R_q)$. The exact sparsity assumption has been widely used for establishing statistical recovery bounds, while the soft sparsity assumption attracts relatively little attention. Throughout this paper, we fix $q\in [0,1]$, and assume that the true parameter $\beta^*\in \B_q(R_q)$ unless otherwise specified.

Now for the purpose of estimating $\beta^*$ based on the observed data $Z_1^m$, many researchers proposed to consider the regularized $M$-estimator (see, e.g., \cite{agarwal2012fast, negahban2012unified}), which is formulated as
\begin{equation}\label{M-esti}
\hat{\beta} \in \argmin_{\beta\in \Omega\subseteq \R^n}\{\loss_m(\beta;Z_1^m)+\regu_\lambda(\beta)\},
\end{equation}
where $\lambda>0$ is a user-defined regularization parameter, and $\regu_\lambda:\R^n\to \R$ is a regularizer depending on $\lambda$ and is assumed to be separable across coordinates, written as $\regu_\lambda(\beta)=\sum_{j=1}^n\rho_\lambda(\beta_j)$ with the decomposable component $\rho_\lambda:\R\to \R$ specified in the following. The loss function $\loss_m$ is required to be differentiable, but do not need to be convex. The regularizer $\regu_\lambda$, which serves to impose certain type of sparsity constraint on the estimator, can also be nonconvex. Due to this potential nonconvexity, we include a side constraint $g:\R^n\to \R_+$, which is required to be convex and satisfy
\begin{equation}\label{g-l1}
g(\beta)\geq \omega\|\beta\|_1,\quad \forall\beta\in \R^n
\end{equation}
for some positive number $\omega>0$. The feasible region is then specialized as
\begin{equation}\label{feasible}
\Omega:=\{\beta\in \R^n: g(\beta)\leq r\}.
\end{equation}
The parameter $r>0$ must be chosen carefully to ensure $\beta^*\in \Omega$. Any point $\beta\in \Omega$ will also satisfy $\|\beta||_1\leq r/\omega$, and provided that $\loss_m$ and $\regu_\lambda$ are continuous, it is guaranteed by the Weierstrass extreme value theorem that a global solution $\hat{\beta}$ always exists. Hereinafter in order to ease the notation, we adopt the shorthand $\loss_m(\cdot)$ for $\loss_m(\cdot;Z_1^m)$.

\subsection{Nonconvex loss function and restricted strong convexity/smoothness}

Throughout this paper, the loss function $\loss_m$ is required to be differentiable with respect to $\beta$, but needs not to be convex. Instead, some weaker conditions known as restricted strong convexity (RSC) and restricted strong smoothness (RSM) are required, which have been discussed precisely in former literature \citep{agarwal2012fast, loh2015regularized}. Specifically, the RSC/RSM conditions imposed on the loss function $\loss_m$ are the same as those used in \cite{loh2015regularized}, thus we here only provide the expressions so as to make this paper complete. For more detailed discussions, see \cite{loh2015regularized}.

We begin with defining the first-order Taylor series expansion around a vector $\beta'$ in the direction of $\beta$ as
\begin{equation}\label{taylor}
\T(\beta,\beta'):=\loss_m(\beta)-\loss_m(\beta')-\langle \nabla\loss_m(\beta'),\beta-\beta' \rangle.
\end{equation}
Then concretely speaking, the RSC condition takes two types of forms, one is used for the analysis of statistical recovery bounds, defined as
\begin{subequations}\label{sta-RSC}
\begin{numcases}{\langle \nabla\loss_m(\beta^*+\delta)-\nabla\loss_m(\beta^*),\delta \rangle\geq}
\gamma_1\|\delta\|_2^2-\tau_1\frac{\log n}{m}\|\delta\|_1^2,\quad \forall\|\delta\|_2\leq 3,\label{sta-RSC1}\\
\gamma_2\|\delta\|_2-\tau_2\sqrt{\frac{\log n}{m}}\|\delta\|_1,\quad \forall\|\delta\|_2>3,\label{sta-RSC2}
\end{numcases}
\end{subequations}
where $\gamma_i\ (i=1,2)$ are positive constants and $\tau_i\ (i=1,2)$ are nonnegative constants; the other one is used for the analysis of algorithmic convergence rate, defined in terms of the Taylor series error \eqref{taylor}
\begin{subequations}\label{alg-RSC}
\begin{numcases}{\T(\beta,\beta')\geq}
\gamma_3\|\beta-\beta'\|_2^2-\tau_3\frac{\log n}{m}\|\beta-\beta'\|_1^2,\quad \forall\|\beta-\beta'\|_2\leq 3,\label{alg-RSC1}\\
\gamma_4\|\beta-\beta'\|_2-\tau_4\sqrt{\frac{\log n}{m}}\|\beta-\beta'\|_1,\quad \forall\|\beta-\beta'\|_2>3,\label{alg-RSC2}
\end{numcases}
\end{subequations}
where $\gamma_i\ (i=3,4)$ are positive constants and $\tau_i\ (i=3,4)$ are nonnegative constants. In addition, the RSM condition is also defined by the Taylor series error \eqref{taylor} as follows:
\begin{equation}\label{alg-RSM}
\T(\beta,\beta')\leq \gamma_5\|\beta-\beta'\|_2^2+\tau_5\frac{\log n}{m}\|\beta-\beta'\|_1^2,\quad \forall \beta,\beta'\in \R^n,
\end{equation}
where $\gamma_5$ is a positive constant and $\tau_5$ is a nonnegative constant.

\subsection{Nonconvex regularizer and regularity conditions}

Now we impose some regularity conditions on the nonconvex regularizer, which are defined in terms of the decomposable component $\rho_\lambda:\R\to \R$.
\begin{Assumption}\mbox{}\par\label{asup-regu}
\begin{enumerate}[\rm(i)]
\item $\rho_\lambda$ satisfies $\rho_\lambda(0)=0$ and is symmetric around zero, that is, $\rho_\lambda(t)=\rho_\lambda(-t)$ for all $t\in \R$;
\item On the nonnegative real line, $\rho_\lambda$ is nondecreasing;
\item For $t>0$, the function $t\mapsto \frac{\rho_\lambda(t)}{t}$ is nonincreasing in $t$;
\item $\rho_\lambda$ is differentiable for all $t\neq 0$ and subdifferentiable at $t=0$, with $\lim\limits_{t\to 0^+}\rho'_\lambda(t)=\lambda L$;
\item $\rho_\lambda$ is subadditive, that is, $\rho_\lambda(t+t')\leq \rho_\lambda(t)+\rho_\lambda(t')$ for all $t,t'\in \R$;
\item $\rho_\lambda$ can be decomposed as $\rho_\lambda(\cdot)=h_\lambda(\cdot)+q_\lambda(\cdot)$, where $h_\lambda$ is convex, and $q_\lambda$ is concave with $q_\lambda(0)=q'_\lambda(0)=0$, $q_\lambda(t)=q_\lambda(-t)$ for all $t\in \R$, and for $t>t'$, there exists two constants $\mu_1\geq 0$ and $\mu_2\geq 0$ such that
    \begin{equation}\label{cond-qlambda}
    -\mu_1\leq \frac{q'_\lambda(t)-q'_\lambda(t')}{t-t'}\leq -\mu_2\leq 0.
    \end{equation}
\end{enumerate}
\end{Assumption}
Conditions (i)-(iv) are the same as those proposed in \cite{loh2015regularized}, and we here explicitly add the condition of subadditivity, though it is relatively mild and are satisfied by a wide range of regularizers. Note that the last condition is a generalization of the weak convexity assumption \cite[Assumption 1(v)]{loh2015regularized}: There exists $\mu>0$, such that $\rho_{\lambda,\mu}(t):=\rho_\lambda(t)+\frac{\mu}{2}t^2$ is convex.

As we will see in the next section, one of the main advantage of adopting condition (vi) in Assumption \ref{asup-regu} is that the proposed algorithm to solve the optimization problem \eqref{M-esti} highly depends on the decomposition in Assumption \ref{asup-regu}(vi) for the regularizer. For general regularizers, condition (vi) provides us the potential to consider different decompositions so as to construct different estimators as well as the iterations.
In particular, as is shown in Example \ref{decomp}, besides the natural decomposition for the SCAD and MCP regularizer inspired by \citet[Assumption 1(v)]{loh2015regularized}, there exists another simpler decomposition, which leads to iterations with simpler forms. Moreover, it is easy to construct functions that satisfy our condition (vi) while do not satisfy \citet[Assumption 1(v)]{loh2015regularized}, but we omit the construction here as many commonly-used regularizers such as the $\ell_1$-norm regularizer $\lambda\|\cdot\|_1$ (Lasso), SCAD and MCP satisfy both our condition (vi) and \citet[Assumption 1(v)]{loh2015regularized}.

It is easy to check that the Lasso regularizer satisfies all these conditions in Assumption \ref{asup-regu}. Other nonconvex regularizers such as the SCAD and MCP regularizers are also contained in our framework. More precisely, it has been shown in \cite{loh2015regularized} that both the SCAD and MCP satisfy conditions (i)-(v) with $L=1$ for condition (iv). To verify condition (vi), in the following, we provide an example showing two different decompositions.
\begin{Example}\label{decomp}
\indent Consider the SCAD regularizer:
\begin{equation*}
\rho_\lambda(t):=\left\{
\begin{array}{l}
\lambda|t|,\ \ \text{if}\ \  |t|\leq \lambda,\\
-\frac{t^2-2a\lambda|t|+\lambda^2}{2(a-1)},\ \  \text{if}\ \ \lambda<|t|\leq a\lambda,\\
\frac{(a+1)\lambda^2}{2},\ \ \text{if}\ \ |t|>a\lambda,
\end{array}
\right.
\end{equation*}
where $a>2$ is a fixed parameter, and the MCP regularizer:
\begin{equation*}
\rho_\lambda(t):=\left\{
\begin{array}{l}
\lambda|t|-\frac{t^2}{2b},\ \ \text{if}\ \  |t|\leq b\lambda,\\
\frac{b\lambda^2}{2},\ \ \text{if}\ \ |t|>b\lambda,
\end{array}
\right.
\end{equation*}
where $b>0$ is a fixed parameter.
To verify condition (vi), the first way is to set
\begin{equation}\label{SCAD-h-1}
h_\lambda(t)=\left\{
\begin{array}{l}
\lambda|t|+\frac{t^2}{2(a-1)},\ \ \text{if}\ \  |t|\leq \lambda,\\
\frac{2a\lambda|t|-\lambda^2}{2(a-1)},\ \  \text{if}\ \ \lambda<|t|\leq a\lambda,\quad q_\lambda(t)=-\frac{t^2}{2(a-1)}\\
\frac{t^2}{2(a-1)}+\frac{(a+1)\lambda^2}{2},\ \ \text{if}\ \ |t|>a\lambda,
\end{array}
\right.
\end{equation}
for SCAD, and
\begin{equation}\label{MCP-h-1}
h_\lambda(t)=\left\{
\begin{array}{l}
\lambda|t|,\ \ \text{if}\ \  |t|\leq b\lambda,\\
\frac{t^2}{2b}+\frac{b\lambda^2}{2},\ \ \text{if}\ \ |t|>b\lambda,\quad q_\lambda(t)=-\frac{t^2}{2b}
\end{array}
\right.
\end{equation}
for MCP, then both the two regularizers satisfy condition (vi). In fact, this decomposition is inspired by \citet[Assumption 1(v)]{loh2015regularized}.
The other way is to set $h_\lambda(\cdot)=\lambda|\cdot|$ for both SCAD and MCP, then
\begin{equation}\label{SCAD-q-2}
q_\lambda(t)=\left\{
\begin{array}{l}
0,\ \ \text{if}\ \  |t|\leq \lambda,\\
-\frac{t^2-2\lambda|t|+\lambda^2}{(2(a-1)},\ \  \text{if}\ \ \lambda<|t|\leq a\lambda,\\
\frac{(a+1)\lambda^2}{2}-\lambda|t|,\ \ \text{if}\ \ |t|>a\lambda,
\end{array}
\right.
\end{equation}
for SCAD with $\mu_1=\frac{1}{a-1}$ and $\mu_2=0$, and
\begin{equation}\label{MCP-q-2}
q_\lambda(t)=\left\{
\begin{array}{l}
-\frac{t^2}{2b},\ \ \text{if}\ \  |t|\leq b\lambda,\\
\frac{b\lambda^2}{2}-\lambda|t|,\ \  \text{if}\ \ |t|> b\lambda,
\end{array}
\right.
\end{equation}
for MCP with $\mu_1=\frac{1}{b}$ and $\mu_2=0$, respectively. Hence, both the SCAD and MCP regularizers satisfy our condition (vi) in Assumption \eqref{asup-regu}.
\end{Example}
We shall see that the second decomposition plays an important role in constructing iterations with more simple forms, so as to be solved more efficiently, with SCAD and MCP as regularizers in the next section.

Now for notational simplicity we define
\begin{equation}\label{HQ-lambda}
\h_\lambda(\cdot):=\sum_{j=1}^{n}h_\lambda(\cdot)\quad \mbox{and}\quad
\Q_\lambda(\cdot):=\sum_{j=1}^{n}q_\lambda(\cdot),
\end{equation}
that is, $\h_\lambda$ and $\Q_\lambda$ denote the decomposable convex component and concave component of the nonconvex regularizer $\regu_\lambda$, respectively.

\section{Main results}\label{sec-main}

In this section, we establish our main results and proofs including statistical guarantee and algorithmic convergence rate. We begin with several lemmas, which are beneficial to the proofs of Theorems \ref{thm-sta} and \ref{thm-algo}. Recall the true parameter $\beta^*\in \B_q(R_q)$. Then for any positive number $\eta>0$, we define the set corresponding to $\beta^*$:
\begin{equation}\label{S-eta}
S_\eta:=\{j\in\{1,2,\cdots,n\}:|\beta^*_j|>\eta\}.
\end{equation}
Then, by a standard argument (see, e.g., \cite{negahban2012unified}), one checks that
\begin{equation}\label{s-eta}
|S_\eta|\leq \eta^{-q}R_q\quad \mbox{and}\quad \|\beta^*_{S_\eta^c}\|_1\leq\eta^{1-q}R_q.
\end{equation}
\begin{Lemma}\label{lem-regu}
Let $\beta,\ \delta\in \R^n$ and $S\subseteq \{1,2,\cdots,n\}$ be any subset with $|S|=s$.
Let $J\subseteq \{1,2,\cdots,n\}$ be the index set of the $s$ largest elements of $\delta$ in absolute value.
Then one has that
\begin{equation}\label{lem-regu-1}
\regu_\lambda(\beta)-\regu_\lambda(\beta+\delta)\leq
\lambda L(\|\delta_J\|_1-\|\delta_{J^c}\|_1)+2\lambda L\|\beta_{S^c}\|_1.
\end{equation}
\end{Lemma}
\begin{proof}
By the decomposiability and the subadditivity of the regularizer $\regu_\lambda$, one has that
\begin{equation}\label{lem-regu-2}
\begin{aligned}
\regu_\lambda(\beta)-\regu_\lambda(\beta+\delta)
&\leq \regu_\lambda(\delta_S)+\regu_\lambda(\beta_{S^c})-\regu_\lambda(\beta_{S^c}+\delta_{S^c})\\ &\leq \regu_\lambda(\delta_S)-\regu_\lambda(\delta_{S^c})+2\regu_\lambda(\beta_{S^c})\\
&\leq \regu_\lambda(\delta_J)-\regu_\lambda(\delta_{J^c})+2\regu_\lambda(\beta_{S^c}),
\end{aligned}
\end{equation}
where the last inequality is from the definition of the set $J$. Then it follows from \citet[Lemma 6]{loh2013local} (with $A=J$ and $k=s$) and \citet[Lemma 4(a)]{loh2015regularized} that
$$\regu_\lambda(\delta_J)-\regu_\lambda(\delta_{J^c})\leq \lambda L(\|\delta_J\|_1-\|\delta_{J^c}\|_1),$$
$$\regu_\lambda(\beta_{S^c})\leq \lambda L\|\beta_{S^c}\|_1.$$
Combining these two inequalities with \eqref{lem-regu-2}, we obtain \eqref{lem-regu-1}. The proof is complete.
\end{proof}
The following two lemmas tell us some general properties of $\h_\lambda$ and $\Q_\lambda$ defined in \eqref{HQ-lambda}, respectively.
\begin{Lemma}\label{lem-hlambda}
Let $\h_\lambda$ be defined in \eqref{HQ-lambda}. Then it holds that
\begin{equation*}
\h_\lambda(\beta)\geq \lambda L\|\beta\|_1,\quad \forall\beta\in \R^n.
\end{equation*}
\begin{proof}
It suffices to show that for all $t\in \R$,
\begin{equation}\label{H-lambda-1}
h_\lambda(t)\geq \lambda L|t|.
\end{equation}
When $t=0$, \eqref{H-lambda-1} follows trivially by Assumption \ref{asup-regu}. To consider the case when $t\not =0$, by the symmetry, we may assume, without loss of generality, that $t>0$. Then since $h_\lambda$ is convex, one has that
for any $t'\in (0,t)$,
\begin{equation*}
\frac{h_\lambda(t)-h_\lambda(0)}{t-0}\geq h'_\lambda(t')=\rho'_\lambda(t')-q'_\lambda(t').
\end{equation*}
Taking $t'\to 0^+$, we have that \eqref{H-lambda-1} holds.
The proof is complete.
\end{proof}
\end{Lemma}
\begin{Lemma}\label{lem-qlambda}
Let $\Q_\lambda$ be defined in \eqref{HQ-lambda}. Then for any $\beta,\beta'\in \R^n$, the following relations are true:
\begin{subequations}
\begin{align}
\langle \nabla\Q_\lambda(\beta)-\nabla\Q_\lambda(\beta'),\beta-\beta' \rangle \geq -\mu_1\|\beta-\beta'\|_2^2,\label{lem-qlambda-11}\\
\langle \nabla\Q_\lambda(\beta)-\nabla\Q_\lambda(\beta'),\beta-\beta' \rangle \leq -\mu_2\|\beta-\beta'\|_2^2, \label{lem-qlambda-12}\\
\Q_\lambda(\beta)\geq \Q_\lambda(\beta')+\langle \nabla\Q_\lambda(\beta'),\beta-\beta' \rangle- \frac{\mu_1}{2}\|\beta-\beta'\|_2^2,\label{lem-qlambda-13}\\
\Q_\lambda(\beta)\leq \Q_\lambda(\beta')+\langle \nabla\Q_\lambda(\beta'),\beta-\beta' \rangle- \frac{\mu_2}{2}\|\beta-\beta'\|_2^2.\label{lem-qlambda-14}
\end{align}
\end{subequations}
\end{Lemma}
\begin{proof}
By \eqref{cond-qlambda}, we have that for any $j=1,2,\cdots,n$,
\begin{equation*}
-\mu_1(\beta_j-\beta_j')^2\leq (q'_\lambda(\beta_j)-q'_\lambda(\beta_j'))(\beta_j-\beta_j')\leq -\mu_2(\beta_j-\beta_j')^2,
\end{equation*}
from which \eqref{lem-qlambda-11} and \eqref{lem-qlambda-12} follow directly. Then by \citet[Theorem 2.1.5 and Theorem 2.1.9]{nesterov2013introductory}, it follows from \eqref{lem-qlambda-11} and \eqref{lem-qlambda-12} that
the convex function $-\Q_\lambda(\beta)$ satisfies
\begin{align*}
-\Q_\lambda(\beta)&\leq -\Q_\lambda(\beta')+\langle \nabla(-\Q_\lambda(\beta')),\beta-\beta' \rangle+ \frac{\mu_1}{2}\|\beta-\beta'\|_2^2,\\
-\Q_\lambda(\beta)&\geq -\Q_\lambda(\beta')+\langle \nabla(-\Q_\lambda(\beta')),\beta-\beta' \rangle- \frac{\mu_2}{2}\|\beta-\beta'\|_2^2,
\end{align*}
which respectively implies that the function $\Q_\lambda(\beta)$ satisfies \eqref{lem-qlambda-13} and \eqref{lem-qlambda-14}.
The proof is complete.
\end{proof}

\subsection{Statistical results}

Recall that the feasible region $\Omega$ is specified in \eqref{feasible}. We shall provide the recovery bound for each stationary point $\tilde{\beta}\in \Omega$ of the optimization problem \eqref{M-esti}, that is, $\tilde{\beta}$ satisfies the first-order necessary condition:
\begin{equation}\label{1st-cond}
\langle \nabla\loss_m(\tilde{\beta})+\nabla\regu_\lambda(\tilde{\beta}),\beta-\tilde{\beta} \rangle\geq 0,\quad \text{for all}\ \beta\in \Omega.
\end{equation}
\begin{Theorem}\label{thm-sta}
Let $R_q>0$ and $r>0$ be positive numbers such that $\beta^*\in \B_q(R_q)\cap \Omega$. Let $\tilde{\beta}$ be a stationary point of the optimization problem \eqref{M-esti}. Suppose that the empirical loss function $\loss_m$ satisfies the RSC conditions \eqref{sta-RSC}, and $\gamma_1>\frac{2\mu_1-\mu_2}{2}$.
Assume that the regularization parameter $\lambda$ is chosen to satisfy
\begin{equation}\label{thm1-lambda}
\frac{2}{L}\max\left\{\|\nabla\loss_m(\beta^*)\|_\infty, \gamma_2\sqrt{\frac{\log n}{m}}
\right\}\leq \lambda\leq \frac{\gamma_2\omega}{2rL},
\end{equation}
and the sample size satisfies
\begin{equation}\label{thm1-m}
m\geq \frac{16r^2\max(\tau_1^2,\tau_2^2)}{\gamma_2^2\omega^2}\log n.
\end{equation}
Then we have that
\begin{equation}\label{l2-rate}
\|\tilde{\beta}-\beta^*\|_2^2\leq (\sqrt{57}+7)^2 R_q\left(\frac{2\lambda L}{2\gamma_1-2\mu_1+\mu_2}\right)^{2-q},
\end{equation}
\begin{equation}\label{l1-rate}
\|\tilde{\beta}-\beta^*\|_1\leq 4(2\sqrt{57}+15)R_q\left(\frac{2\lambda L}{2\gamma_1-2\mu_1+\mu_2}\right)^{1-q}.
\end{equation}
\end{Theorem}
\begin{proof}
Set $\tilde{\delta}:=\tilde{\beta}-\beta^*$. We first show that $\|\tilde{\delta}\|_2\leq 3$. Suppose on the contrary that $\|\tilde{\delta}\|_2>3$. Then one has the following inequality by \eqref{sta-RSC2}:
\begin{equation}\label{thm1-1}
\langle \nabla\loss_m(\tilde{\beta})-\nabla\loss_m(\beta^*),\tilde{\delta} \rangle\geq \gamma_2\|\tilde{\delta}\|_2-\tau_2\sqrt{\frac{\log n}{m}}\|\tilde{\delta}\|_1.
\end{equation}
Noting $\beta^*\in \Omega$, and combining \eqref{thm1-1} and \eqref{1st-cond} (with $\beta^*$ in place of $\beta$), we arrive at
\begin{equation}\label{thm1-2}
\langle -\nabla\regu_\lambda(\tilde{\beta})-\nabla\loss_m(\beta^*),\tilde{\delta} \rangle\geq \gamma_2\|\tilde{\delta}\|_2-\tau_2\sqrt{\frac{\log n}{m}}\|\tilde{\delta}\|_1.
\end{equation}
Applying H{\"o}lder's inequality and the triangle inequality to the left-hand side of \eqref{thm1-2}, and noting that $\regu_\lambda$ satisfies Assumption \ref{asup-regu}, one has by \citet[Lemma 4]{loh2015regularized} and \eqref{thm1-lambda} that
\begin{equation*}
\langle -\nabla\regu_\lambda(\tilde{\beta})-\nabla\loss_m(\beta^*),\tilde{\delta} \rangle
\leq \{\|\nabla\regu_\lambda(\tilde{\beta})\|_\infty+\|\nabla\loss_m(\beta^*)\|_\infty\}\|\tilde{\delta}\|_1\\ \leq \left\{\lambda L+\frac{\lambda L}{2}\right\}\|\tilde{\delta}\|_1.
\end{equation*}
Then combining this inequality with \eqref{thm1-2} and noting that $\|\tilde{\delta}\|_1\leq \|\tilde{\beta}\|_1+\|\beta^*\|_1\leq g(\tilde{\beta})/\omega+g(\beta^*)/\omega\leq 2r/\omega$
(due to \eqref{g-l1}),
we obtain that
\begin{equation*}
\|\tilde{\delta}\|_2\leq \frac{\|\tilde{\delta}\|_1}{\gamma_2}\left(\frac{3\lambda L}{2}+\tau_2\sqrt{\frac{\log n}{m}}\right)\leq \frac{2r}{\gamma_2\omega}\left(\frac{3\lambda L}{2}+\tau_2\sqrt{\frac{\log n}{m}}\right).
\end{equation*}
Since $\lambda$ satisfies \eqref{thm1-lambda} and $m$ satisfies \eqref{thm1-m},
we obtain that $\|\tilde{\delta}\|_2\leq 3$, a contradiction. Thus, $\|\tilde{\delta}\|_2\leq 3$.
Therefore, by \eqref{sta-RSC1}, one has that
\begin{equation}\label{thm1-3}
\langle \nabla\loss_m(\tilde{\beta})-\nabla\loss_m(\beta^*),\tilde{\delta} \rangle\geq \gamma_1\|\tilde{\delta}\|_2^2-\tau_1\frac{\log n}{m}\|\tilde{\delta}\|_1^2.
\end{equation}
On the other hand, it follows from \eqref{lem-qlambda-11} and  \eqref{lem-qlambda-14} in Lemma \ref{lem-qlambda} that
\begin{equation*}
\begin{aligned}
\langle \nabla\regu_\lambda(\tilde{\beta}),\beta^*-\tilde{\beta} \rangle
&= \langle \nabla\Q_\lambda(\tilde{\beta})+\nabla\h_\lambda(\tilde{\beta}),\beta^*-\tilde{\beta} \rangle \\
&\leq \langle \nabla\Q_\lambda(\beta^*),\beta^*-\tilde{\beta} \rangle+\mu_1\|\beta^*-\tilde{\beta}\|_2^2+\langle \nabla\h_\lambda(\tilde{\beta}),\beta^*-\tilde{\beta} \rangle \\
&\leq \Q_\lambda(\beta^*)-\Q_\lambda(\tilde{\beta})+\frac{2\mu_1-\mu_2}{2}\|\beta^*-\tilde{\beta}\|_2^2+\langle \nabla\h_\lambda(\tilde{\beta}),\beta^*-\tilde{\beta} \rangle.
\end{aligned}
\end{equation*}
Moreover, since the function $\h_\lambda$ is convex, one has that
\begin{equation*}
\h_\lambda(\beta^*)-\h_\lambda(\tilde{\beta})\geq \langle \nabla\h_\lambda(\tilde{\beta}),\beta^*-\tilde{\beta} \rangle.
\end{equation*}
This, together with the former inequality, implies that
\begin{equation}\label{thm1-6}
\langle \nabla\regu_\lambda(\tilde{\beta}),\beta^*-\tilde{\beta} \rangle \leq \regu_\lambda(\beta^*)-\regu_\lambda(\tilde{\beta})+\frac{2\mu_1-\mu_2}{2}\|\beta^*-\tilde{\beta}\|_2^2.
\end{equation}
Then combining \eqref{thm1-3}, \eqref{thm1-6} and \eqref{1st-cond} (with $\beta^*$ in place of $\beta$), we obtain that
\begin{equation}\label{thm1-9}
\begin{aligned}
\gamma_1\|\tilde{\delta}\|_2^2-\tau_1\frac{\log n}{m}\|\tilde{\delta}\|_1^2
&\leq -\langle \nabla\loss_m(\beta^*),\tilde{\delta} \rangle+\regu_\lambda(\beta^*)-\regu_\lambda(\tilde{\beta})+\frac{2\mu_1-\mu_2}{2}\|\beta^*-\tilde{\beta}\|_2^2\\
&\leq \|\nabla\loss_m(\beta^*)\|_\infty\|\tilde{\delta}\|_1+\regu_\lambda(\beta^*)-\regu_\lambda(\tilde{\beta})+\frac{2\mu_1-\mu_2}{2}\|\tilde{\delta}\|_2^2.
\end{aligned}
\end{equation}
Let $J$ denote the index set corresponding to the $|S_\eta|$ largest coordinates in
absolute value of $\tilde{\delta}$ (recalling the set $S_\eta$ defined in \eqref{S-eta}. It then follows from Lemma \ref{lem-regu} (with $S=S_\eta$) that
\begin{equation}\label{thm1-8}
\regu_\lambda(\beta^*)-\regu_\lambda(\tilde{\beta})\leq
\lambda L(\|\tilde{\delta}_J\|_1-\|\tilde{\delta}_{J^c}\|_1)+2\lambda L\|\beta^*_{S_\eta^c}\|_1.
\end{equation}
Then combining \eqref{thm1-8} with \eqref{thm1-9} and noting assumption \eqref{thm1-lambda}, one has that
\begin{equation*}
\begin{aligned}
\gamma_1\|\tilde{\delta}\|_2^2-\tau_1\frac{\log n}{m}\|\tilde{\delta}\|_1^2
&\leq \|\nabla\loss_m(\beta^*)\|_\infty\|\tilde{\delta}\|_1+\lambda L(\|\tilde{\delta}_J\|_1-\|\tilde{\delta}_{J^c}\|_1)+2\lambda L\|\beta^*_{S_\eta^c}\|_1+\frac{2\mu_1-\mu_2}{2}\|\tilde{\delta}\|_2^2\\
&\leq \frac{3\lambda L}{2}\|\tilde{\delta}_J\|_1-\frac{\lambda L}{2}\|\tilde{\delta}_{J^c}\|_1+2\lambda L\|\beta^*_{S_\eta^c}\|_1+\frac{2\mu_1-\mu_2}{2}\|\tilde{\delta}\|_2^2.
\end{aligned}
\end{equation*}
This, together with the fact $\|\tilde{\delta}\|_1\leq 2r/\omega$ and assumptions \eqref{thm1-lambda} and \eqref{thm1-m}, implies that
\begin{equation}\label{thm1-15}
\begin{aligned}
2(\gamma_1-\frac{2\mu_1-\mu_2}{2})\|\tilde{\delta}\|_2^2
&\leq 3\lambda L\|\tilde{\delta}_J\|_1-\lambda L\|\tilde{\delta}_{J^c}\|_1+4\tau_1\frac{r}{\omega}\frac{\log n}{m}\|\tilde{\delta}\|_1+2\lambda L\|\beta^*_{S_\eta^c}\|_1\\
&\leq 3\lambda L\|\tilde{\delta}_J\|_1-\lambda L\|\tilde{\delta}_{J^c}\|_1+\gamma_2\sqrt{\frac{\log n}{m}}\|\tilde{\delta}\|_1+2\lambda L\|\beta^*_{S_\eta^c}\|_1\\
&\leq \frac{7\lambda L}{2}\|\tilde{\delta}_J\|_1-\frac{\lambda L}{2}\|\tilde{\delta}_{J^c}\|_1+2\lambda L\|\beta^*_{S_\eta^c}\|_1.
\end{aligned}
\end{equation}
Since $\gamma_1>\frac{2\mu_1-\mu_2}{2}$ by assumption, we have by \eqref{thm1-15} that $\|\tilde{\delta}_{J^c}\|_1\leq 7\|\tilde{\delta}_J\|_1+4\|\beta^*_{S_\eta^c}\|_1$. Consequently,
\begin{equation}\label{thm1-12}
\|\tilde{\delta}\|_1\leq
8\|\tilde{\delta}_J\|_1+4\|\beta^*_{S_\eta^c}\|_1\leq 8\sqrt{|S_\eta|}\|\tilde{\delta}_J\|_2+4\|\beta^*_{S_\eta^c}\|_1
\leq 8\sqrt{|S_\eta|}\|\tilde{\delta}\|_2+4\|\beta^*_{S_\eta^c}\|_1.
\end{equation}
Furthermore, \eqref{thm1-15} implies that
\begin{equation}\label{thm1-13}
2\left(\gamma_1-\frac{2\mu_1-\mu_2}{2}\right)\|\tilde{\delta}\|_2^2 \leq \frac{7\lambda L}{2}\|\tilde{\delta}_J\|_1+2\lambda L\|\beta^*_{S_\eta^c}\|_1
\leq 28\lambda L\sqrt{|S_\eta|}\|\tilde{\delta}\|_2+16\lambda L\|\beta^*_{S_\eta^c}\|_1.
\end{equation}
Combining \eqref{s-eta} with \eqref{thm1-13} and setting $\eta=\frac{\lambda L}{\gamma_1-\frac{2\mu_1-\mu_2}{2}}$, we obtain \eqref{l2-rate}.
Moreover, it follows from \eqref{thm1-12} that \eqref{l1-rate} holds. The proof is complete.
\end{proof}
\begin{Remark}
{\rm (i)} Theorem \ref{thm-sta} tells us that the $\ell_2$ recovery bound for all the stationary points of the nonconvex optimization problem \eqref{M-esti} scales as $\|\tilde{\beta}-\beta^*\|_2^2=O(\lambda^{2-q}R_q)$. When $\lambda$ is chosen as $\lambda=\Omega\left(\sqrt{\frac{\log n}{m}}\right)$, the $\ell_2$ recovery bound implies that the estimator $\tilde{\beta}$ is statistically consistent. Note that this result is independent of any  specific algorithms, meaning that any numerical algorithm for solving the nonconvex optimization problem \eqref{M-esti} can stably recover the true sparse parameter as long as it is guaranteed to converge to a stationary point.

{\rm (ii)} In the case when $q=0$, the underlying parameter $\beta^*$ is of exact sparsity with $\|\beta^*\|_0=R_0$, and Theorem \ref{thm-sta} is reduced to \citet[Theorem 1]{loh2015regularized} up to constant factors. More generally, Theorem \ref{thm-sta} provides the $\ell_2$ recovery bound when $\beta^*\in \B_q(R_q)$ with $q\in [0,1]$. Note that the sparsity of $\beta^*$ is measured via the $\ell_q$-norm, with larger values meaning lesser sparsity, \eqref{l2-rate} indicates that the rate of the recovery bound slows down as $q$ increases to 1.
\end{Remark}

\subsection{Algorithmic results}

We now apply the proximal gradient method \citep{nesterov2007gradient} to solve a modified version of the nonconvex optimization problem \eqref{M-esti} and then establish the linear convergence rate under the RSC/RSM conditions. Recall that the regularizer can be decomposed as $\regu_\lambda(\cdot)=\Q_\lambda(\cdot)+\h_\lambda(\cdot)$ by Assumption \ref{asup-regu}. In the following, we shall consider the side constraint function specialized as
\begin{equation*}
g(\cdot):=\frac{1}{\lambda}\h_\lambda(\cdot),
\end{equation*}
which is convex by Assumption \ref{asup-regu} and satisfies $g(\beta)\geq L\|\beta\|_1$ for all $\beta\in \R^n$ by Lemma \ref{lem-hlambda}, meeting our requirement \eqref{g-l1} with $\omega=L$. The optimization problem \eqref{M-esti} now can be written as
\begin{equation}\label{M-esti-alg}
\hat{\beta} \in \argmin_{\beta\in \R^n, g(\beta)\leq r}\{\bar{\loss}_m(\beta)+\h_\lambda(\beta)\},
\end{equation}
with $\bar{\loss}_m(\cdot):=\loss_m(\cdot)+\Q_\lambda(\cdot)$. By means of this decomposition, the objective function is decomposed into a differentiable but possibly nonconvex function and a possibly nonsmooth but convex function. Applying the proximal gradient method proposed in \cite{nesterov2007gradient} to \eqref{M-esti-alg}, we obtain a sequence of iterates $\{\beta^t\}_{t=0}^\infty$ as
\begin{equation}\label{algo-pga}
\beta^{t+1}\in \argmin_{\beta\in \R^n, g(\beta)\leq r}\left\{\frac{1}{2}\Big{\|}\beta-\left(\beta^t-\frac{\nabla{\bar{\loss}_m(\beta^t)}}{v}\right)\Big{\|}_2^2+\frac{1}{v}\h_\lambda(\beta)\right\},
\end{equation}
where $\frac{1}{v}$ is the step size.

Given $\beta^t$, one can follow \cite{loh2015regularized} to generate the next iterate $\beta^{t+1}$ via the following three steps; see \citet[Appendix C.1]{loh2015regularized} for details.
\begin{enumerate}[(1)]
\item First optimize the unconstrained optimization problem
\begin{equation*}
\hat{\beta^t}\in \argmin_{\beta\in \R^n}\left\{\frac{1}{2}\Big{\|}\beta-\left(\beta^t-\frac{\nabla{\bar{\loss}_m(\beta^t)}}{v}\right)\Big{\|}_2^2+\frac{1}{v}\h_\lambda(\beta)\right\}.
\end{equation*}
\item If $g(\hat{\beta^t})\leq r$, define $\beta^{t+1}=\hat{\beta^t}$.
\item Otherwise, if $g(\hat{\beta^t})> r$, optimize the constrained optimization problem
\begin{equation*}
\beta^{t+1}\in \argmin_{\beta\in \R^n, g(\beta)\leq r}\left\{\frac{1}{2}\Big{\|}\beta-\left(\beta^t-\frac{\nabla{\bar{\loss}_m(\beta^t)}}{v}\right)\Big{\|}_2^2\right\}.
\end{equation*}
\end{enumerate}
\indent For specific regularizers such as SCAD and MCP, one could consider two different decompositions for the regularizer as we did in Example \ref{decomp}. Particularly, if we use the first decomposition,
then $\h_\lambda$ is a piecewise function (cf. \eqref{SCAD-h-1}, \eqref{MCP-h-1}, and \eqref{HQ-lambda}), and implementing iteration \eqref{algo-pga} may require large computational cost. However, if we use the second decomposition,
then \eqref{algo-pga} turns to
\begin{equation}\label{algo-pga-1}
\beta^{t+1}\in \argmin_{\beta\in \R^n, \|\beta\|_1\leq r}\left\{\frac{1}{2}\Big{\|}\beta-\left(\beta^t-\frac{\nabla{\bar{\loss}_m(\beta^t)}}{v}\right)\Big{\|}_2^2+\frac{\lambda}{v}\|\beta\|_1\right\},
\end{equation}
corresponding to first performing the soft-threshold operator and then performing $\ell_2$ projection onto the $\ell_1$-ball of radius $r$, which can be computed rapidly in $\mathcal{O}(n)$ time using a procedure proposed in \cite{Duchi2008Efficient}. The advantage of iteration \eqref{algo-pga-1} is due to the more general condition (vi) in Assumption \ref{asup-regu}. We shall further compare these two decompositions by simulations in section \ref{sec-simul}.

Before we state our main result that the algorithm defined by \eqref{algo-pga} converges linearly to a small neighbourhood of any global solution $\hat{\beta}$, we shall need some notations to simplify our expositions.
Let $\phi(\cdot):=\loss_m(\cdot)+\regu_\lambda(\cdot)=\bar{\loss}_m(\cdot)+\h_\lambda(\cdot)$ denote the optimization objective function. The Taylor error $\bar{\T}(\beta,\beta')$ for the modified loss function $\bar{\loss}_m$ is defined as follows:
\begin{equation}\label{taylor-modified}
\bar{\T}(\beta,\beta')=\T(\beta,\beta')+\Q_\lambda(\beta)-\Q_\lambda(\beta')-\langle \nabla\Q_\lambda(\beta'),\beta-\beta' \rangle.
\end{equation}
Recall the RSC and RSM conditions in \eqref{alg-RSC} and \eqref{alg-RSM}, respectively. Throughout this section, we assume $2\gamma_i>\mu_1$ for all $i=3,4,5$, and set $\gamma:=\min\{\gamma_3,\gamma_4\}$ and $\tau:=\max\{\tau_3,\tau_4,\tau_5\}$. Recall that the true underlying parameter $\beta^*\in \B_q(R_q)$ (cf. \eqref{lq-ball}).  Let $\hat{\beta}$ be a global solution of the optimization problem \eqref{M-esti}. Then unless otherwise specified, we define
\begin{align}
&\bar{\epsilon}_{\text{stat}}:=8R_q^{\frac12}\left(\frac{\log n}{m}\right)^{-\frac{q}{4}}\left(\|\hat{\beta}-\beta^*\|_2+R_q^{\frac12}\left(\frac{\log n}{m}\right)^{\frac{1}{2}-\frac{q}{4}}\right),\label{bar-epsilon}\\
&\kappa:= \left\{1-\frac{2\gamma-\mu_1}{8v}+\frac{256R_q\tau\left(\frac{\log n}{m}\right)^{1-\frac{q}2}}{2\gamma-\mu_1}\right\}\left\{1-\frac{256R_q\tau\left(\frac{\log n}{m}\right)^{1-\frac{q}2}}{2\gamma-\mu_1}\right\}^{-1},\label{lem-bound-kappa}\\
&\xi:= 2\tau\frac{\log n}{m}\left\{\frac{2\gamma-\mu_1}{8v}+\frac{512R_q\tau\left(\frac{\log n}{m}\right)^{1-\frac{q}2}}{2\gamma-\mu_1}+5\right\}\left\{1-\frac{256R_q\tau\left(\frac{\log n}{m}\right)^{1-\frac{q}2}}{2\gamma-\mu_1}\right\}^{-1}.\label{lem-bound-xi}
\end{align}
For a given number $\Delta>0$ and an integer $T>0$ such that
\begin{equation}\label{lem-cone-De1}
\phi(\beta^t)-\phi(\hat{\beta})\leq \Delta, \quad \forall\ t\geq T,
\end{equation}
define
\begin{equation*}
\epsilon(\Delta):=\frac{2}{L}\min\left(\frac{\Delta}{\lambda},r\right).
\end{equation*}
With this setup, we now state our main algorithmic result.
\begin{Theorem}\label{thm-algo}
Let $R_q>0$ and $r>0$ be positive numbers such that $\beta^*\in \B_q(R_q)\cap \Omega$. Let $\hat{\beta}$ be a global solution of the optimization problem \eqref{M-esti}.
Suppose that the empirical loss function $\loss_m$ satisfies the RSC/RSM conditions \eqref{alg-RSC} and \eqref{alg-RSM}, and $\gamma>\frac{1}{2}\mu_1$. Let $\{\beta^t\}_{t=0}^\infty$ be a sequence of iterates generated via \eqref{algo-pga} with an initial point $\beta^0$ satisfying $\|\beta^0-\hat{\beta}\|_2\leq 3$ and step size $v\geq \max\{2\gamma_5-\mu_2,\mu_1\}$.
Assume that the regularization paramter $\lambda$ is chosen to satisfy
\begin{equation}\label{thm2-lambda}
\frac{4}{L}\max\left\{\|\nabla\loss_m(\beta^*)\|_\infty,
\tau\sqrt{\frac{\log n}{m}}\right\}\leq \lambda\leq \frac{6\gamma-9\mu_1}{4r},
\end{equation}
and the sample size satisfies
\begin{equation}\label{thm2-m}
m\geq \max\left\{\frac{4r^2}{L^2}, \left(\frac{128R_q\tau}{2\gamma-\mu_1}\right)^{1-\frac{q}{2}}\right\}\log n.
\end{equation}
Then for any tolerance $\Delta^*\geq\frac{8\xi}{1-\kappa}\bar{\epsilon}_{\emph{stat}}^2$ and any iteration $t\geq T(\Delta^*)$, we have that
\begin{equation}\label{thm2-error}
\|\beta^t-\hat{\beta}\|_2^2\leq \left(\frac{4}{2\gamma-\mu_1}\right)\left(\Delta^*+\frac{{\Delta^*}^2}{2\tau}+4\tau\frac{\log n}{m}\bar{\epsilon}_{\emph{stat}}^2\right),
\end{equation}
where
\begin{equation*}
\begin{aligned}
T(\Delta^*)&:=\log_2\log_2\left(\frac{r\lambda}{\Delta^*}\right)\left(1+\frac{\log 2}{\log(1/\kappa)}\right) +\frac{\log((\phi(\beta^0)-\phi(\hat{\beta}))/\Delta^*)}{\log(1/\kappa)},
\end{aligned}
\end{equation*}
and $\bar{\epsilon}_{\emph{stat}}$, $\kappa$, $\xi$ are defined in
\eqref{bar-epsilon}-\eqref{lem-bound-xi}, respectively.
\end{Theorem}
\begin{Remark}
{\rm (i)} Note that in the case when $q=0$, the underlying parameter $\beta^*$ is exact sparse with $\|\beta^*\|_0=R_0$, and Theorem \ref{thm-algo} is reduced to \citet[Theorem 2]{loh2015regularized} up to constant factors.\\
{\rm (ii)}  Generally speaking, Theorem \ref{thm-algo} has established the linear convergence rate when $\beta^*\in \B_q(R_q)$ with  $q\in [0,1]$ and pointed out some significant differences between the case of exact sparsity and that of soft sparsity. Specifically, it is ensured that the algorithm in \citet[Theorem 2]{loh2015regularized} converges linearly to a small neighbourhood of the global solution $\hat{\beta}$ and the optimization error only depends on the statistical recovery bound $\|\hat{\beta}-\beta^*\|_2$. In contrast, besides the statistical error $\|\hat{\beta}-\beta^*\|_2$, our optimization error \eqref{thm2-error}
in the case when $q\in (0,1]$ also has an additional term $R_q\left(\frac{\log n}{m}\right)^{1-\frac q2}$ (cf. \eqref{bar-epsilon}), which appears because of the statistical nonidentifiability over the $\ell_q$-ball, and it is no larger than $\|\hat{\beta}-\beta^*\|_2$
with overwhelming probability.
\end{Remark}

Before providing the proof of Theorem \ref{thm-algo}, we give several useful lemmas with the corresponding proofs deferred to Appendix.
\begin{Lemma}\label{lem-radius}
Under the conditions of Theorem \ref{thm-algo}, it holds that for all $t\geq 0$
\begin{equation}\label{lem-radius-1}
\|\beta^t-\hat{\beta}\|_2\leq 3.
\end{equation}
\end{Lemma}
\begin{Lemma}\label{lem-cone}
Suppose that the conditions of Theorem \ref{thm-algo} are satisfied, and
that there exists a pair $(\Delta, T)$ such that \eqref{lem-cone-De1} holds.
Then for any iteration $t\geq T$, it holds that
\begin{equation*}
\begin{aligned}
\|\beta^t-\hat{\beta}\|_1&\leq 4R_q^{\frac12}\left(\frac{\log n}{m}\right)^{-\frac{q}{4}}\|\beta^t-\hat{\beta}\|_2+\bar{\epsilon}_{\emph{stat}}+\epsilon(\Delta).
\end{aligned}
\end{equation*}
\end{Lemma}
\begin{Lemma}\label{lem-Tphi-bound}
Suppose that the conditions of Theorem \ref{thm-algo} are satisfied and that there exists a pair $(\Delta, T)$ such that \eqref{lem-cone-De1} holds.
Then for any iteration $t\geq T$, we have that
\begin{align}
\bar{\T}(\hat{\beta},\beta^t)&\geq -2\tau\frac{\log n}{m}(\bar{\epsilon}_{\emph{stat}}^2+\epsilon^2(\Delta))^2,\label{lem-bound-T1}\\
\phi(\beta^t)-\phi(\hat{\beta})&\geq \left(\frac{2\gamma-\mu_1}{4}\right)\|\beta^t-\hat{\beta}\|_2^2-2\tau\frac{\log n}{m}(\bar{\epsilon}_{\emph{stat}}^2+\epsilon^2(\Delta))^2, \label{lem-bound-T2}\\
\phi(\beta^t)-\phi(\hat{\beta})&\leq \kappa^{t-T}(\phi(\beta^T)-\phi(\hat{\beta}))+\frac{2\xi}{1-\kappa}(\bar{\epsilon}_{\emph{stat}}^2+\epsilon^2(\Delta)).\label{lem-bound-0}
\end{align}
\end{Lemma}

By virtue of the above lemmas, we are now ready to prove Theorem \ref{thm-algo}.
\begin{proof}(Proof of Theorem \ref{thm-algo})
We first prove the inequality as follows:
\begin{equation}\label{thm2-1}
\phi(\beta^t)-\phi(\hat{\beta})\leq \Delta^*,\quad \forall t\geq T(\Delta^*).
\end{equation}
Divide the iterations $t=0,1,\cdots$ into a series of disjoint epochs $[T_k,T_{k+1}]$ and define an associated sequence of tolerances $\Delta_0>\Delta_1>\cdots$ such that
\begin{equation*}
\phi(\beta^t)-\phi(\hat{\beta})\leq \Delta_k,\quad \forall t\geq T_k,
\end{equation*}
as well as the associated error term $\epsilon_k:=\frac{2}{L}\min \left\{\frac{\Delta_k}{\lambda},r\right\}$. The values of $\{(\Delta_k,T_k)\}_{k\geq 1}$ will be chosen later.
Then at the first iteration,
we apply Lemma \ref{lem-Tphi-bound} (cf. \eqref{lem-bound-0}) with $\epsilon_0=2r/L$ and $T_0=0$
to conclude that
\begin{equation}\label{thm2-T0}
\phi(\beta^t)-\phi(\hat{\beta})\leq \kappa^t(\phi(\beta^0)-\phi(\hat{\beta}))+\frac{2\xi}{1-\kappa}(\bar{\epsilon}_{\text{stat}}^2 +\frac{4r^2}{L^2}),\quad \forall t\geq T_0.
\end{equation}
Set $\Delta_1:=\frac{4\xi}{1-\kappa}(\bar{\epsilon}_{\text{stat}}^2 +\frac{4r^2}{L^2})$. Noting that $\kappa\in (0,1)$ by assumption, one has by \eqref{thm2-T0} that for $T_1:=\lceil \frac{\log (2\Delta_0/\Delta_1)}{\log (1/\kappa)} \rceil$,
\begin{equation*}
\begin{aligned}
\phi(\beta^t)-\phi(\hat{\beta})&\leq \frac{\Delta_1}{2}+\frac{2\xi}{1-\kappa}\left(\bar{\epsilon}_{\text{stat}}^2 +\frac{4r^2}{L^2}\right)=\Delta_1\leq \frac{8\xi}{1-\kappa}\max\left\{\bar{\epsilon}_{\text{stat}}^2,\frac{4r^2}{L^2}\right\},\quad \forall t\geq T_1.
\end{aligned}
\end{equation*}
For $k\geq 1$, we define
\begin{equation}\label{thm2-DeltaT}
\Delta_{k+1}:= \frac{4\xi}{1-\kappa}(\bar{\epsilon}_{\text{stat}}^2+\epsilon_k^2)
\quad \mbox{and}\quad
T_{k+1}:= \left\lceil \frac{\log (2\Delta_k/\Delta_{k+1})}{\log (1/\kappa)}+T_k \right\rceil.
\end{equation}
Then Lemma \ref{lem-Tphi-bound} (cf. \eqref{lem-bound-0}) is applicable to concluding that for all $t\geq T_k$,
\begin{equation*}
\phi(\beta^t)-\phi(\hat{\beta})\leq \kappa^{t-T_k}(\phi(\beta^{T_k})-\phi(\hat{\beta}))+\frac{2\xi}{1-\kappa}(\bar{\epsilon}_{\text{stat}}^2+\epsilon_k^2),
\end{equation*}
which implies that
\begin{equation*}
\phi(\beta^t)-\phi(\hat{\beta})\leq \Delta_{k+1}\leq \frac{8\xi}{1-\kappa}\max\{\bar{\epsilon}_{\text{stat}}^2,\epsilon_k^2\},\quad \forall t\geq T_{k+1}.
\end{equation*}
From \eqref{thm2-DeltaT}, we obtain the following recursion for $\{(\Delta_k,T_k)\}_{k=0}^\infty$
\begin{subequations}
\begin{align}
\Delta_{k+1}&\leq \frac{8\xi}{1-\kappa}\max\{\epsilon_k^2,\bar{\epsilon}_{\text{stat}}^2\},\label{thm2-recur-Delta}\\
T_k&\leq k+\frac{\log (2^k\Delta_0/\Delta_k)}{\log (1/\kappa)}.\label{thm2-recur-T}
\end{align}
\end{subequations}
Then by \cite[Section 7.2]{agarwal2012supplementaryMF}, one sees that \eqref{thm2-recur-Delta} implies that
\begin{equation}\label{thm2-recur2}
\Delta_{k+1}\leq \frac{\Delta_k}{4^{2^{k+1}}}\quad \mbox{and}\quad \frac{\Delta_{k+1}}{\lambda}\leq \frac{r}{4^{2^k}},\quad \forall k\geq 1.
\end{equation}
Now let us show how to decide the smallest $k$ such that $\Delta_k\leq \Delta^*$ by applying \eqref{thm2-recur2}. If we are in the first epoch, \eqref{thm2-1} is clearly from \eqref{thm2-recur-Delta}.
Otherwise, by \eqref{thm2-recur-T}, we see that $\Delta_k\leq \Delta^*$ holds after at most $$k(\Delta^*)\geq \frac{\log(\log(r\lambda/\Delta^*)/\log 4)}{\log(2)}+1=\log_2\log_2(r\lambda/\Delta^*)$$ epoches.
Combining the above bound on $k(\Delta^*)$ with \eqref{thm2-recur-T}, we conclude
that $\phi(\beta^t)-\phi(\hat{\beta})\leq \Delta^*$ holds for all iterations
\begin{equation*}
t\geq
\log_2\log_2\left(\frac{r\lambda}{\Delta^*}\right)\left(1+\frac{\log 2}{\log(1/\kappa)}\right)+\frac{\log(\Delta_0/\Delta^*)}{\log(1/\kappa)},
\end{equation*}
which proves \eqref{thm2-1}.
Finally, as \eqref{thm2-1} is proved, it follows from \eqref{lem-bound-T2} in Lemma \ref{lem-Tphi-bound}
and assumption \eqref{thm2-m} that, for any $t\geq T(\Delta^*)$,
\begin{equation*}
\left(\frac{2\gamma-\mu_1}{4}\right)\|\beta^t-\hat{\beta}\|_2^2
\leq \phi(\beta^t)-\phi(\hat{\beta}) +2\tau\frac{\log n}{m}\left(\epsilon(\Delta^*)+\bar{\epsilon}_{\text{stat}}\right)^2
\leq \Delta^*+2\tau\frac{\log n}{m}\left(\frac{2\Delta^*}{\lambda L}+\bar{\epsilon}_{\text{stat}}\right)^2.
\end{equation*}
Consequently,
by assumption \eqref{thm2-lambda}, we conclude that for any $t\geq T(\Delta^*)$,
\begin{equation*}
\|\beta^t-\hat{\beta}\|_2^2\leq \left(\frac{4}{2\gamma-\mu_1}\right)\left(\Delta^*+\frac{{\Delta^*}^2}{2\tau}+4\tau\frac{\log n}{m}\bar{\epsilon}_{\text{stat}}^2\right).
\end{equation*}
The proof is complete.
\end{proof}

\section{Simulations on the corrupted linear regression model}\label{sec-simul}

In this section, we carry out several numerical experiments to illustrate our theoretical results and compare
the performance of the estimators obtained by two different decompositions for the regularizer. Specifically, we consider high-dimensional linear regression with corrupted observations. Recall the standard linear regression model
\begin{equation}\label{ordi-linear}
y_i=\langle \beta^*,X_{i\cdot} \rangle+e_i,\quad \text{for}\ i=1,2\cdots,m,
\end{equation}
where $\beta^*\in \R^n$ is the unknown parameter and $\{(X_{i\cdot},y_i)\}_{i=1}^m$ are i.i.d. observations, which are assumed to be fully-observed in standard formulations. However, this assumption is not realistic for many applications, in which the covariates may be observed only partially and one can only observe the pairs $\{(Z_{i\cdot},y_i)\}_{i=1}^m$ instead, where $Z_{i\cdot}$'s are corrupted versions of the corresponding $X_{i\cdot}$'s. As has been discussed in \cite{loh2012high, loh2015regularized}, there are mainly two types of corruption:
\begin{enumerate}[(a)]
\item Additive noise: For each $i=1,2,\cdots,m$, we observe $Z_{i\cdot}=X_{i\cdot}+W_{i\cdot}$, where $W_{i\cdot}\in \R^n$ is a random vector independent of $X_{i\cdot}$ with mean \textbf{0} and known covariance matrix $\Sigma_w$.
\item Missing data: For each $i=1,2,\cdots,m$, we observe a random vector $Z_{i\cdot}\in \R^n$, such that for each $j=1,2\cdots,n$, we independently observe $Z_{ij}=X_{ij}$ with probability $1-\vartheta$, and $Z_{ij}=0$ with probability $\vartheta$, where $\vartheta\in [0,1)$.
\end{enumerate}

Following a line of past works \citep{loh2012high, loh2015regularized}, we fix $i\in \{1,2,\cdots,m\}$ and use $\Sigma_x$ to denote the covariance matrix of the covariates $X_{i\cdot}$ (i.e., $\Sigma_x=\text{cov}(X_{i\cdot})$). The population loss function is
$\loss(\beta) =\frac12\beta^\top\Sigma_x\beta-{\beta^*}^\top\Sigma_x\beta$.
Let $(\hat{\Gamma},\hat{\Upsilon})$ denote the estimators for $(\Sigma_x,\Sigma_x\beta^*)$ that depend only on the observed data $\{(Z_{i\cdot},y_i)\}_{i=1}^m$, and the empirical loss function is then written as
\begin{equation}\label{empi-corr}
\loss_m(\beta)=\frac12\beta^\top\hat{\Gamma}\beta-\hat{\Upsilon}^\top\beta.
\end{equation}
Substituting the empirical loss function \eqref{empi-corr} into \eqref{M-esti}, and recalling the side constraint \eqref{g-l1} as well as the feasible region $\Omega$ \eqref{feasible}, the following optimization problem is used to estimate $\beta^*$ in the corrupted linear regression
\begin{equation*}
\hat{\beta} \in \argmin_{\beta\in \Omega}\left\{\left(\frac12\beta^\top\hat{\Gamma}\beta-\hat{\Upsilon}^\top\beta\right)+\regu_\lambda(\beta)\right\}.
\end{equation*}

As discussed in \cite{loh2012high}, an appropriate choice of the surrogate pair $(\hat{\Gamma},\hat{\Upsilon})$ for the additive noise and missing data cases is given respectively by
\begin{equation*}
\begin{aligned}
\hat{\Gamma}_{\text{add}}&:= \frac{Z^\top Z}{m}-\Sigma_w \quad \mbox{and}\quad \hat{\Upsilon}_{\text{add}}:=\frac{Z^\top y}{m}, \\
\hat{\Gamma}_{\text{mis}}&:= \frac{\tilde{Z}^\top \tilde{Z}}{m}-\vartheta\cdot \text{diag}\left(\frac{\tilde{Z}^\top \tilde{Z}}{m}\right) \quad \mbox{and}\quad \hat{\Upsilon}_{\text{mis}}:=\frac{\tilde{Z}^\top y}{m}\quad \left(\tilde{Z}=\frac{Z}{1-v}\right).
\end{aligned}
\end{equation*}

The following simulations will be performed with the loss function $\loss_m$ corresponding to linear regression with additive noise and missing data, respectively, and three regularizers, namely the Lasso, SCAD and MCP. All numerical experiments are performed in MATLAB R2014b and executed on a personal desktop (Intel Core i7-4790, 3.60 GHz, 8.00 GB of RAM).

The numerical data are generated as follows. We first generate i.i.d. samples $X_{i\cdot}\sim N(0,\mathbb{I}_n)$ and the noise term $e\sim N(0,(0.1)^2\mathbb{I}_m)$. Then the true parameter $\beta^*$ is generated as a compressible signal whose entries are all nonzeros but obey a power low decay. Specifically, the signal $\beta^*$ is generated by taking a fixed sequence
$\{5.0\times i^{-2}:i=1,2,\cdots,n\}$, multiplying the sequence by a random sign sequence and permuting at random finally. The data $y$ are generated according to \eqref{ordi-linear}. The corrupted term is set to $W_{i\cdot}\sim N(0,(0.2)^2\mathbb{I}_n)$ and $\vartheta=0.2$ for the additive noise and missing data cases, respectively. The problem size $n$ and $m$ will be specified based on the experiments. The data are then generated at random for 100 times. The performance of the $M$-estimator $\hat{\beta}$ is characterized by the relative error $\|\hat{\beta}-\beta^*\|_2/\|\beta^*\|_2$ and is illustrated by averaging across the 100 numerical results.

As we have mentioned in the preceding sections, there are two different decompositions for the SCAD and MCP regularizers, respectively, which result in two specific forms for \eqref{algo-pga} as follows:
\begin{align}
\beta^{t+1}&\in \argmin_{\beta\in \R^n, \frac{1}{\lambda}\h_\lambda(\beta)\leq r}\left\{\frac{1}{2}\Big{\|}\beta-\left(\beta^t-\frac{\nabla{\bar{\loss}_m(\beta^t)}}{v}\right)\Big{\|}_2^2+\frac{1}{v}\h_\lambda(\beta)\right\},\label{simu-pga-1}\\
\beta^{t+1}&\in \argmin_{\beta\in \R^n, \|\beta\|_1\leq r}\left\{\frac{1}{2}\Big{\|}\beta-\left(\beta^t-\frac{\nabla{\bar{\loss}_m(\beta^t)}}{v}\right)\Big{\|}_2^2+\frac{\lambda}{v}\|\beta\|_1\right\},\label{simu-pga-2}
\end{align}
where $\bar{\loss}_m(\cdot)=\loss_m(\cdot)+\Q_\lambda(\cdot)$ with $\Q_\lambda(\cdot)$ and $\h_\lambda(\cdot)$ specified in \eqref{SCAD-h-1}, \eqref{MCP-h-1}, and \eqref{HQ-lambda} for \eqref{simu-pga-1}, and $\Q_\lambda(\cdot)$ given by \eqref{SCAD-q-2}, \eqref{MCP-q-2} and \eqref{HQ-lambda} for \eqref{simu-pga-2}. In the following, we use SCAD$\_1$ and SCAD$\_2$ to denote the estimators obtained by iterations \eqref{simu-pga-1} and \eqref{simu-pga-2} with SCAD as the regularizer, respectively, and use MCP$\_1$ and MCP$\_2$ to denote the estimators produced by iterations \eqref{simu-pga-1} and \eqref{simu-pga-2} with MCP as the regularizer, respectively. Note that for the Lasso regularizer, these two iterations become the same one, and we use Lasso to stand for the estimator produced by either \eqref{simu-pga-1} or \eqref{simu-pga-2} with $\regu_\lambda(\cdot)=\lambda\|\cdot\|_1$.
For all simulations, the regularization parameter is set to $\lambda=\sqrt{\frac{\log n}{m}}$, $r=\frac{1.1}{\lambda}\h_\lambda(\beta^*)$ for \eqref{simu-pga-1} and $r=1.1\|\beta^*\|_1$ for \eqref{simu-pga-2}, to ensure the feasibility of $\beta^*$. Both \eqref{simu-pga-1} and \eqref{simu-pga-2} are implemented with the step size $\frac1{v}=\frac{1}{2\lambda_{\max}(\Sigma_x)}$ and the initial point $\beta^0=\textbf{0}$. We choose the parameter $a=3.7$ for SCAD and $b=1.5$ for MCP.

The first experiment is performed to demonstrate the statistical guarantee for corrupted linear regression in additive noise and missing data cases with Lasso, SCAD and MCP as regularizers, respectively. For the sake of simplification, we here only report results obtained by iteration \eqref{simu-pga-2} though iteration \eqref{simu-pga-1} is also applicable. Fig. \ref{f-stat} plots the relative error versus the rescaled sample size $\frac{m}{\log n}$ for three different vector dimensions $n\in \{256,512,1024\}$. The estimators Lasso, SCAD$\_2$ and MCP$\_2$ are represented by solid, dotted and dashed lines, respectively. We can see from Fig. \ref{f-stat} that the three curves corresponding to the same regularizer in each panel (a) and (b) nearly match with one another under different problem dimensions $n$,
coinciding with Theorem \ref{thm-sta}. Moreover, as the sample size increases, the relative error decreases to zero, implying the statistical consistency of the estimators.
\begin{figure}[htbp]
\centering
\subfigure[]{
\includegraphics[width=0.49\textwidth]{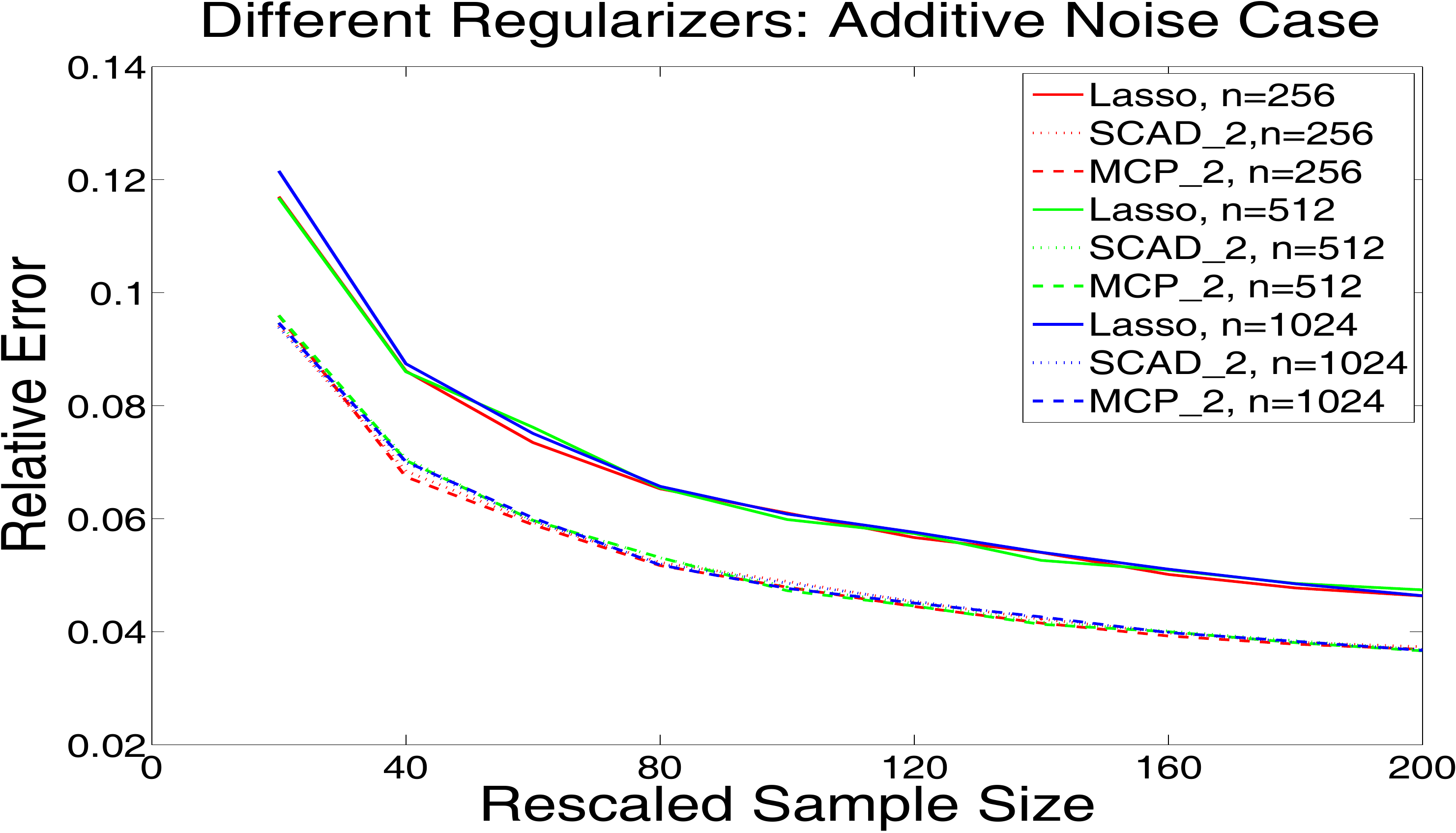}}
\subfigure[]{
\includegraphics[width=0.49\textwidth]{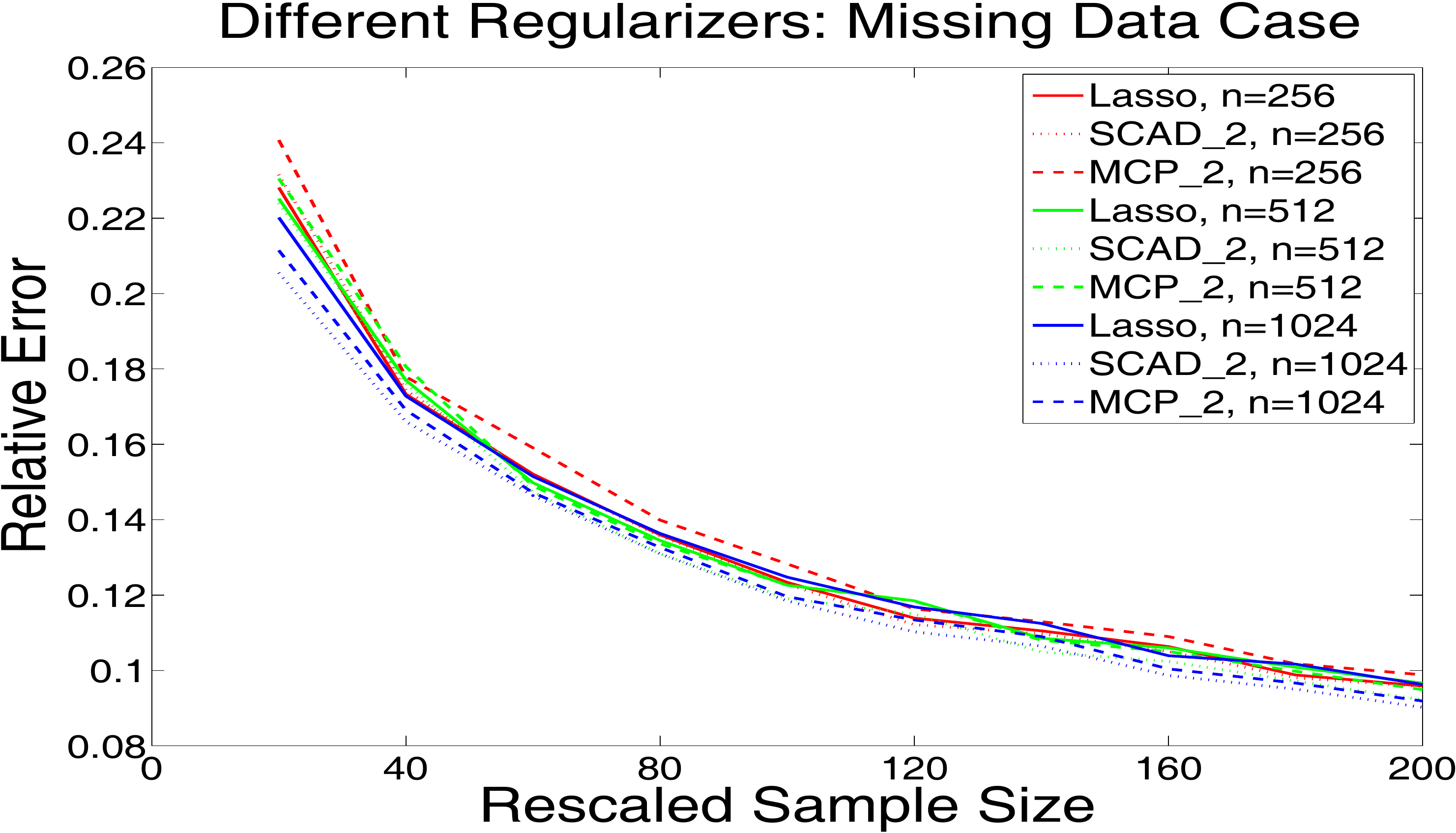}}
\caption{Statistical consistency for corrupted linear regression with Lasso, SCAD and MCP as the regularizers.}
\label{f-stat}
\end{figure}

The second experiment is designed to compare the performance of the estimators produced by two different decompositions for the SCAD and MCP regularizer, namely the estimators obtained by iterations \eqref{simu-pga-1} and \eqref{simu-pga-2}, respectively. We have investigated the performance for a broad range of dimensions $n$ and $m$, and the results are comparatively consistent across these choices. Hence we here report results for $n=1024$ and a range of the sample sizes $m=\lceil\alpha\log n\rceil$ specified by $\alpha\in\{10,30,80\}$.

In the additive noise case, we can see from Fig. \ref{f-add-scad} that for the SCAD regularizer, there seems no difference in the accuracy between the two decompositions across the range of sample sizes.
However, for the MCP regularizer, it is shown in Fig. \ref{f-add-mcp} that estimators obtained by iteration \eqref{simu-pga-2} achieve a more accurate level and a faster convergence rate than those produced by iteration \eqref{simu-pga-1} whenever the sample size is small or large.
\begin{figure}[htbp]
\centering
\subfigure[]{
\label{f-add-scad}
\includegraphics[width=0.48\textwidth]{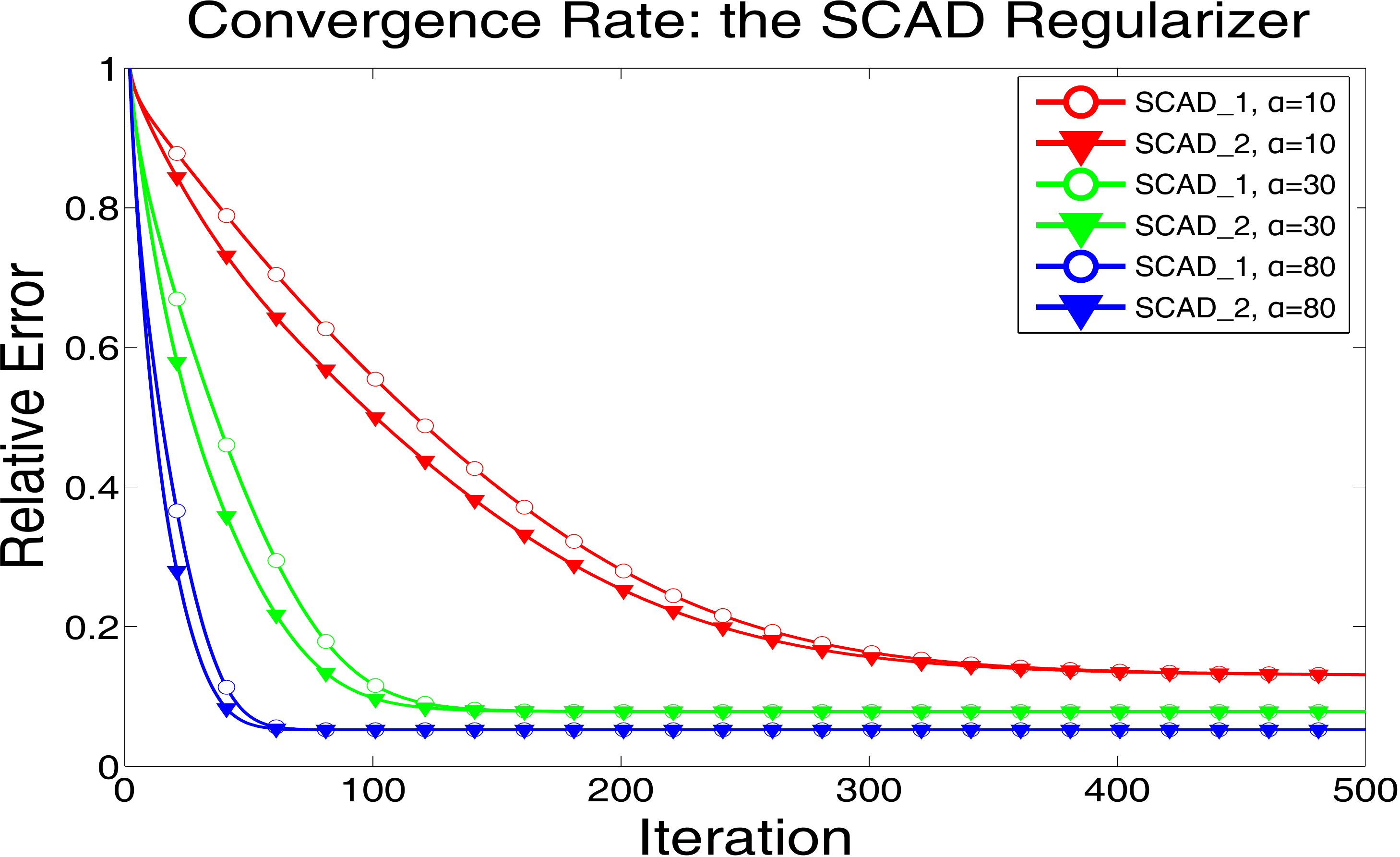}}
\subfigure[]{
\label{f-add-mcp}
\includegraphics[width=0.48\textwidth]{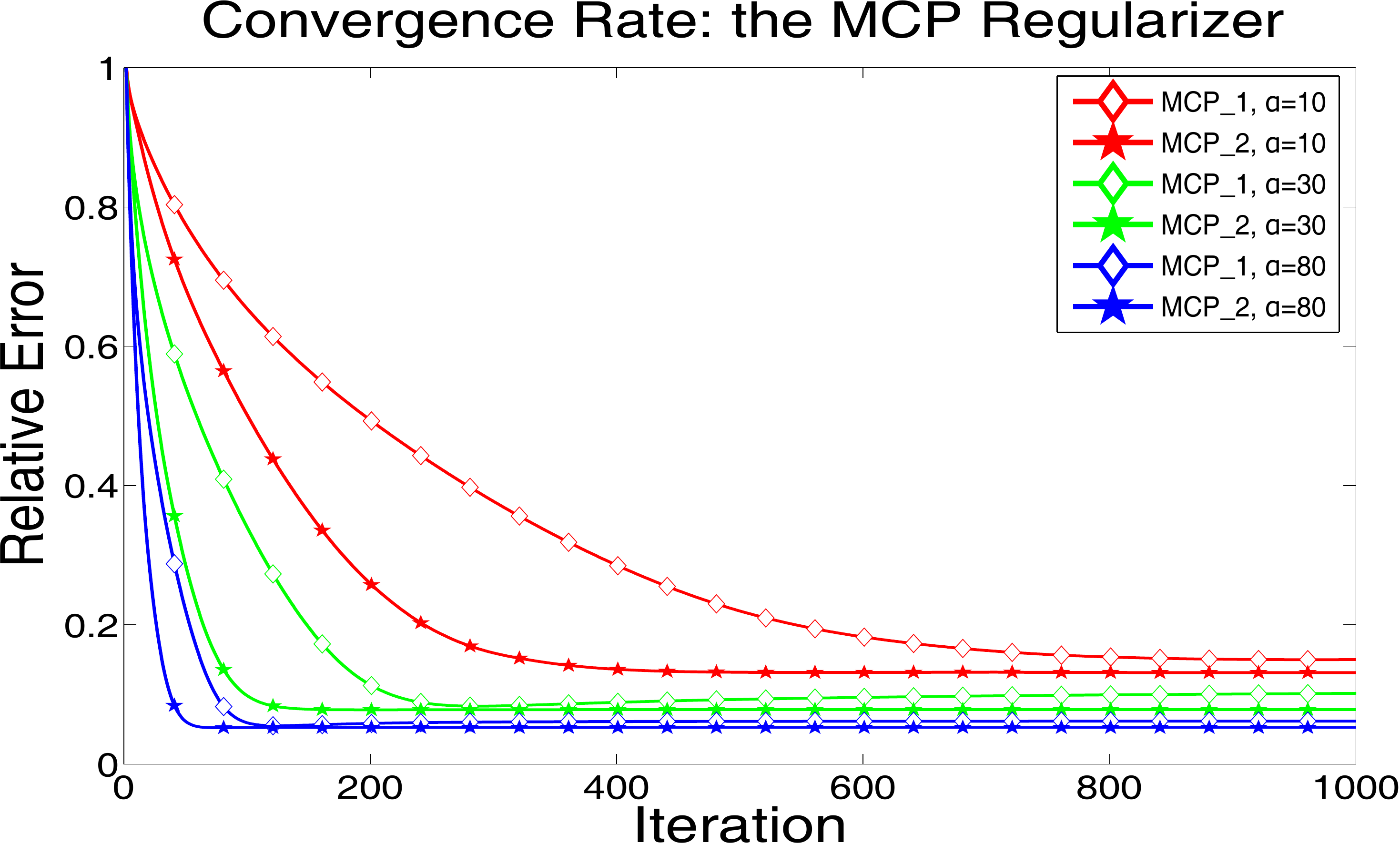}}
\caption{Comparison of decompositions for the SCAD and MCP regularizers in the additive noise case.}
\label{f-com-add}
\end{figure}

Fig. \ref{f-com-mis} depicts analogous results to Fig. \ref{f-com-add} in the case of missing data. For the SCAD regularizer, we can see from Fig. \ref{f-mis-scad} that when the sample size is small (e.g., $\alpha=10$), the difference in accuracy between these two decompositions is slight. Then as the sample size becomes larger (e.g., $\alpha=30,80$), estimators obtained by iteration \eqref{simu-pga-2} achieve a more accurate level and a faster convergence rate than those produced by iteration \eqref{simu-pga-1}. For the MCP regularizer, it is illustrated in Fig. \ref{f-mis-mcp} that estimators obtained by iteration \eqref{simu-pga-2} outperform those obtained by iteration \eqref{simu-pga-1} on both accuracy and convergence rate whenever the sample size is small or large.
\begin{figure}[htbp]
\centering
\subfigure[]{
\label{f-mis-scad}
\includegraphics[width=0.48\textwidth]{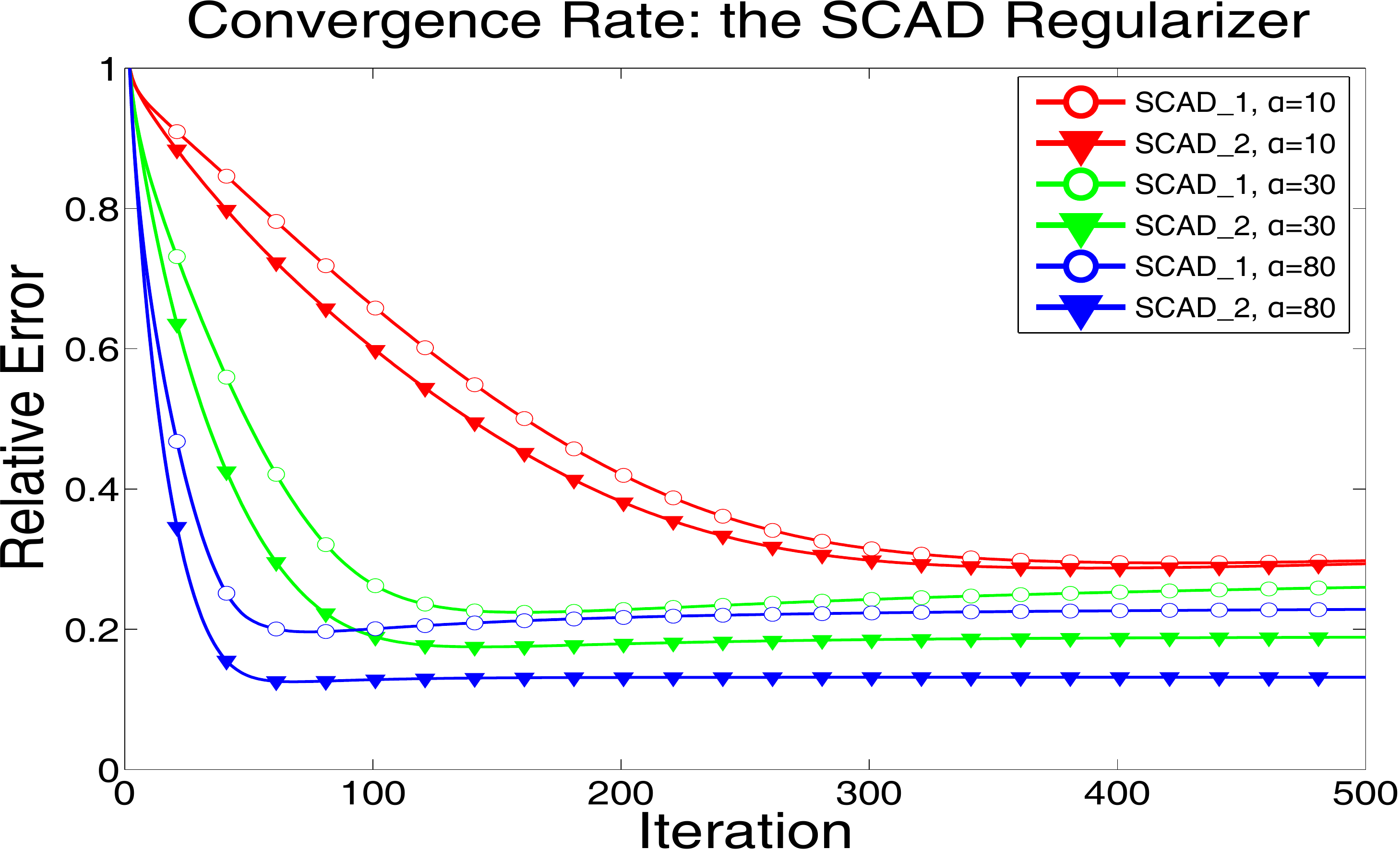}}
\subfigure[]{
\label{f-mis-mcp}
\includegraphics[width=0.48\textwidth]{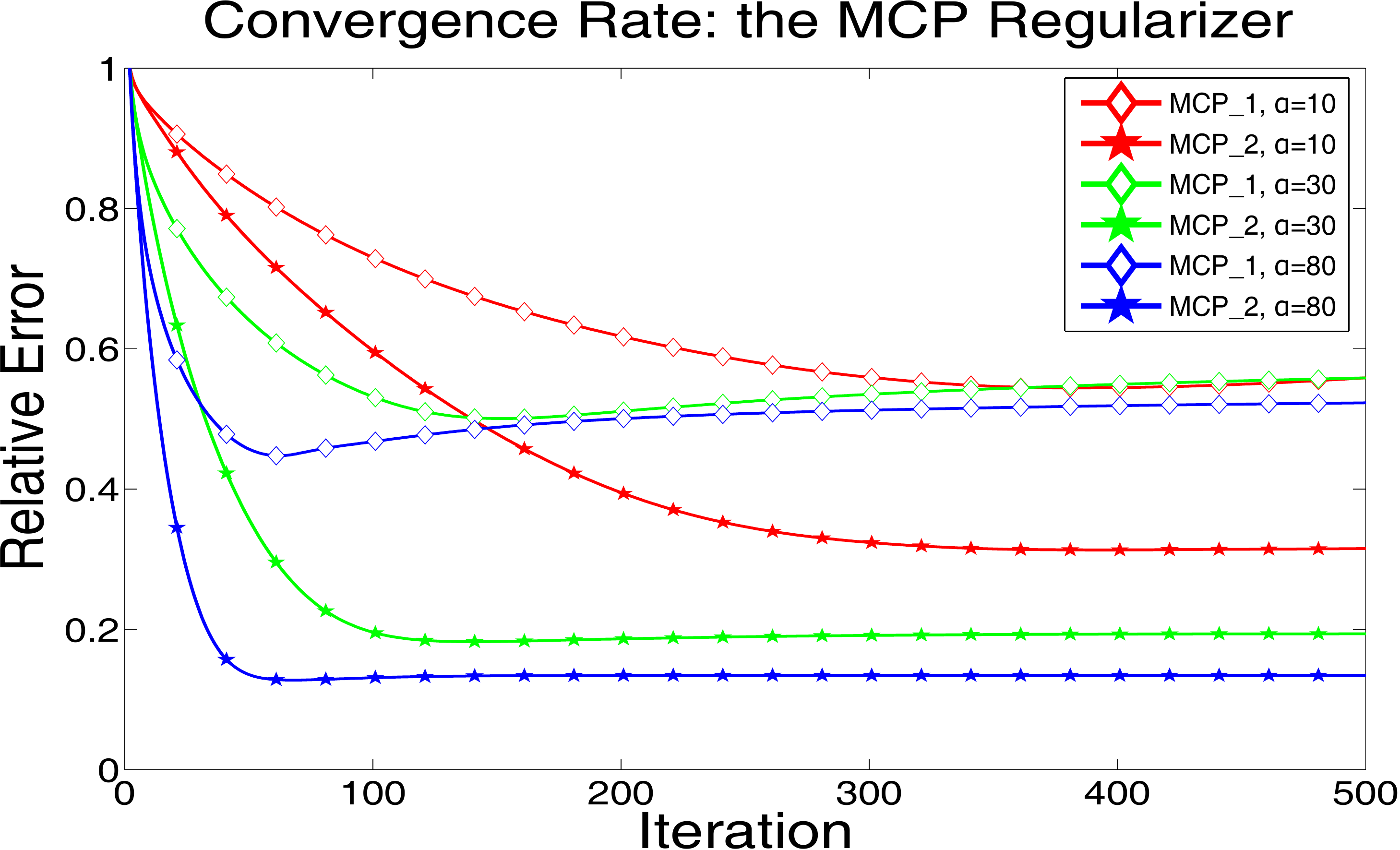}}
\caption{Comparison of decompositions for the SCAD and MCP regularizers in the missing data case.}
\label{f-com-mis}
\end{figure}

In a word, estimators produced by iteration \eqref{simu-pga-2} perform better than those obtained by iteration \eqref{simu-pga-1} on both accuracy and convergence speed, especially in the missing data case. This advantage is due to the more general condition (cf. Assumption \ref{asup-regu}(vi)), which provides the potential to consider different decompositions for specific regularizers and then to design a more efficient algorithm.

\section{Conclusion}\label{sec-concl}

In this work, we investigated the theoretical properties of local solutions of nonconvex regularized $M$-estimators, where the underlying true parameter is assumed to be of soft sparsity. We provided guarantees on statistical consistency for all stationary points of the nonconvex regularized $M$-estimators. Then we applied the proximal gradient algorithm to solve a modified version of the nonconvex optimization problem and established the linear convergence rate. Particularly, for SCAD and MCP, our assumption on the regularizer provides the possibility to consider different decompositions so as to construct estimators with better performance. Finally, the theoretical consequences and the advantage of the assumption on the regularizer were demonstrated by several simulations. However, there exist some other regularizers that do not satisfy our assumptions, such as the bridge regularizers widely used in compressed sensing and machine learning. We are still working to deal with this problem.

\section*{Appendix}\label{sec-appe}

\subsection*{Proof of Lemma \ref{lem-radius}}

The conclusion will be proved by induction on the iteration count $t$. Note that in the base case when $t=0$, the conclusion holds by assumption.
Now in the induction step, let $k\geq 0$ be given and suppose that \eqref{lem-radius-1} holds for $t=k$. Then it remains to show that \eqref{lem-radius-1} holds for $t=k+1$.
Suppose on the contrary that $\|\beta^{k+1}-\hat{\beta}\|_2>3$. By the RSC condition \eqref{alg-RSC2} and \eqref{taylor-modified}, one has that
\begin{equation*}
\bar{\T}(\beta^{k+1},\hat{\beta})\geq \gamma_4\|\beta^{k+1}-\hat{\beta}\|_2-\tau_4\sqrt{\frac{\log n}{m}}\|\beta^{k+1}-\hat{\beta}\|_1+\Q_\lambda(\beta^{k+1})-\Q_\lambda(\hat{\beta})-\langle \nabla\Q_\lambda(\hat{\beta}),\beta^{k+1}-\hat{\beta} \rangle.
\end{equation*}
It then follows from \eqref{lem-qlambda-13} in Lemma \ref{lem-qlambda} that
\begin{equation*}
\begin{aligned}
\bar{\T}(\beta^{k+1},\hat{\beta})&\geq \gamma_4\|\beta^{k+1}-\hat{\beta}\|_2-\tau_4\sqrt{\frac{\log n}{m}}\|\beta^{k+1}-\hat{\beta}\|_1-\frac{\mu_1}{2}\|\beta^{k+1}-\hat{\beta}\|_2^2.
\end{aligned}
\end{equation*}
Moreover, since the function $\h_\lambda$ is convex, one has that
\begin{equation*}
\h_\lambda(\beta^{k+1})-\h_\lambda(\hat{\beta})\geq \langle \nabla\h_\lambda(\hat{\beta}),\beta^{k+1}-\hat{\beta} \rangle.
\end{equation*}
This, together with the former inequality, implies that
\begin{equation*}
\phi(\beta^{k+1})-\phi({\beta})-\langle \nabla\phi(\hat{\beta}),\beta^{k+1}-\hat{\beta} \rangle\geq \gamma_4\|\beta^{k+1}-\hat{\beta}\|_2-\tau_4\sqrt{\frac{\log n}{m}}\|\beta^{k+1}-\hat{\beta}\|_1-\frac{\mu_1}{2}\|\beta^{k+1}-\hat{\beta}\|_2^2.
\end{equation*}
Since $\hat{\beta}$ is the optimal solution, one has by the first-order optimality condition $\langle \nabla\phi(\hat{\beta}),\beta^{k+1}-\hat{\beta} \rangle\geq 0$ that
\begin{equation}\label{lem-radius-3}
\begin{aligned}
\phi(\beta^{k+1})-\phi(\hat{\beta})&\geq \gamma_4\|\beta^{k+1}-\hat{\beta}\|_2-\tau_4\sqrt{\frac{\log n}{m}}\|\beta^{k+1}-\hat{\beta}\|_1-\frac{\mu_1}{2}\|\beta^{k+1}-\hat{\beta}\|_2^2.
\end{aligned}
\end{equation}
On the other hand, since $\|\beta^k-\hat{\beta}\|_2\leq 3$ by the induction hypothesis, applying the RSC condition \eqref{alg-RSC1} on the pair $(\hat{\beta},\beta^k)$, we have by \eqref{lem-qlambda-13} that
\begin{equation*}
\begin{aligned}
\bar{\loss}_m(\hat{\beta})&\geq \bar{\loss}_m(\beta^k)+\langle \nabla\bar{\loss}_m(\beta^k),\hat{\beta}-\beta^k \rangle+\left(\gamma_3-\frac{\mu_1}{2}\right)\|\hat{\beta}-\beta^k\|_2^2-\tau_3\frac{\log n}{m}\|\hat{\beta}-\beta^k\|_1^2.
\end{aligned}
\end{equation*}
This, together with $\h_\lambda(\hat{\beta})\geq \h_\lambda(\beta^{k+1})+\langle \nabla\h_\lambda(\beta^{k+1}),\hat{\beta}-\beta^{k+1} \rangle$
and the assmption that $\gamma>\frac12\mu_1$, implies that
\begin{equation}\label{lem-radius-4}
\begin{aligned}
\phi(\hat{\beta})
&\geq \bar{\loss}_m(\beta^k)+\langle \nabla\bar{\loss}_m(\beta^k),\hat{\beta}-\beta^k \rangle+\h_\lambda(\beta^{k+1})+\langle \nabla\h_\lambda(\beta^{k+1}),\hat{\beta}-\beta^{k+1} \rangle+ \left(\gamma_3-\frac{\mu_1}{2}\right)\|\hat{\beta}-\beta^k\|_2^2-\tau_3\frac{\log n}{m}\|\hat{\beta}-\beta^k\|_1^2\\
&\geq \bar{\loss}_m(\beta^k)+\langle \nabla\bar{\loss}_m(\beta^k),\hat{\beta}-\beta^k \rangle+
\h_\lambda(\beta^{k+1})+\langle \nabla\h_\lambda(\beta^{k+1}),\hat{\beta}-\beta^{k+1} \rangle-\tau_3\frac{\log n}{m}\|\hat{\beta}-\beta^k\|_1^2.
\end{aligned}
\end{equation}
Applying the RSM condition \eqref{alg-RSM} on the pair $(\beta^{k+1},\beta^k)$, one has by \eqref{lem-qlambda-14} and the assumption $v\geq 2\gamma_5-\mu_2$ that
\begin{equation}\label{lem-radius-5}
\begin{aligned}
\phi(\beta^{k+1})
&\leq \bar{\loss}_m(\beta^k)+\langle \nabla\bar{\loss}_m(\beta^k),\beta^{k+1}-\beta^k \rangle+\h_\lambda(\beta^{k+1})+ \left(\gamma_5-\frac{\mu_2}{2}\right)\|\beta^{k+1}-\beta^k\|_2^2+\tau_5\frac{\log n}{m}\|\beta^{k+1}-\beta^k\|_1^2\\
&\leq \bar{\loss}_m(\beta^k)+\langle \nabla\bar{\loss}_m(\beta^k),\beta^{k+1}-\beta^k \rangle+\h_\lambda(\beta^{k+1})+ \frac{v}{2}\|\beta^{k+1}-\beta^k\|_2^2+4\frac{r^2\tau_5}{L^2}\frac{\log n}{m}.
\end{aligned}
\end{equation}
Moreover, it is easy to verify that update \eqref{algo-pga} can be written equivalently as
\begin{equation}\label{lem-radius-8}
\begin{aligned}
\beta^{k+1}\in \argmin_{\beta\in \R^n, g(\beta)\leq r}\left\{\bar{\loss}_m(\beta^k)+\langle \nabla{\bar{\loss}_m(\beta^k)},\beta-\beta^k \rangle+\frac{v}{2}\|\beta-\beta^k\|_2^2+\h_\lambda(\beta)\right\}.
\end{aligned}
\end{equation}
Since $\beta^{k+1}$ is the optimal solution of \eqref{lem-radius-8}, it follows that
\begin{equation}\label{lem-radius-6}
\langle \nabla{\bar{\loss}_m(\beta^k)}+v(\beta^{k+1}-\beta^k)+\nabla\h_\lambda(\beta^{k+1}),\beta^{k+1}-\hat{\beta} \rangle\leq 0.
\end{equation}
Combining \eqref{lem-radius-4}, \eqref{lem-radius-5} and \eqref{lem-radius-6}, one has that
\begin{equation*}
\begin{aligned}
\phi(\beta^{k+1})-\phi(\hat{\beta})
&\leq \frac{v}{2}\|\beta^{k+1}-\beta^k\|_2^2+v\langle \beta^k-\beta^{k+1},\beta^{k+1}-\hat{\beta} \rangle +\tau_3\frac{\log n}{m}\|\hat{\beta}-\beta^k\|_1^2+4\frac{r^2\tau_5}{L^2}\frac{\log n}{m}\\
&\leq \frac{v}{2}\|\beta^k-\hat{\beta}\|_2^2-\frac{v}{2}\|\beta^{k+1}-\hat{\beta}\|_2^2+\tau\frac{\log n}{m}\|\hat{\beta}-\beta^k\|_1^2+4\frac{r^2\tau}{L^2}\frac{\log n}{m}\\
&\leq \frac{v}{2}\|\beta^k-\hat{\beta}\|_2^2-\frac{v}{2}\|\beta^{k+1}-\hat{\beta}\|_2^2+8\frac{r^2\tau}{L^2}\frac{\log n}{m}.
\end{aligned}
\end{equation*}
This, together with \eqref{lem-radius-3} and the assumption $v\geq \mu_1$, implies that
\begin{equation*}
\begin{aligned}
\gamma_4\|\beta^{k+1}-\hat{\beta}\|_2-\tau_4\sqrt{\frac{\log n}{m}}\|\beta^{k+1}-\hat{\beta}\|_1
&\leq
\frac{v}{2}\|\beta^k-\hat{\beta}\|_2^2-\frac{v-\mu_1}{2}\|\beta^{k+1}-\hat{\beta}\|_2^2+8\frac{r^2\tau}{L^2}\frac{\log n}{m}\\
&\leq \frac{9v}{2}-\frac{3(v-\mu_1)}{2}\|\beta^{k+1}-\hat{\beta}\|_2+8\frac{r^2\tau}{L^2}\frac{\log n}{m},
\end{aligned}
\end{equation*}
where the last inequality follows from $\|\beta^k-\hat{\beta}\|_2\leq 3$ by the induction hypothesis and $\|\beta^{k+1}-\hat{\beta}\|_2>3$ by assumption.
Since $\|\beta^{k+1}-\hat{\beta}\|_1\leq \|\beta^{k+1}\|_1+\|\hat{\beta}\|_1\leq g(\beta^{k+1})/L+g(\hat{\beta})/L\leq 2r/L$, it follows that
\begin{equation}\label{lem-radius-contra}
\begin{aligned}
\left(\gamma+\frac{3(v-\mu_1)}{2}\right)\|\beta^{k+1}-\hat{\beta}\|_2
&\leq \frac{9v}{2}+\tau\sqrt{\frac{\log n}{m}}\|\beta^{k+1}-\hat{\beta}\|_1+8\frac{r^2\tau}{L^2}\frac{\log n}{m}\\
&\leq \frac{9v}{2}+2\frac{r\tau}{L}\sqrt{\frac{\log n}{m}}+8\frac{r^2\tau}{L}\frac{\log n}{m}.
\end{aligned}
\end{equation}
By assumptions \eqref{thm2-lambda} and \eqref{thm2-m},
we obtain that $2\frac{r\tau}{L}\sqrt{\frac{\log n}{m}}\leq \frac32(\gamma-\frac{3\mu_1}{2})$ and that $8\frac{r^2\tau}{L^2}\frac{\log n}{m}\leq \frac32(\gamma-\frac{3\mu_1}{2})$. Combining these two inequalities with \eqref{lem-radius-contra}, one has that
\begin{equation*}
\left(\gamma+\frac{3(v-\mu_1)}{2}\right)\|\beta^{k+1}-\hat{\beta}\|_2
\leq 3\left(\gamma+\frac{3(v-\mu_1)}{2}\right),
\end{equation*}
Hence, $\|\beta^{k+1}-\hat{\beta}\|_2\leq 3$, a contradiction. Thus, \eqref{lem-radius-1} holds for $t=k+1$.
By the principle of induction, \eqref{lem-radius-1} holds for all $t\geq 0$. The proof is complete.

\subsection*{Proof of Lemma \ref{lem-cone}}

We first show that if $\lambda\geq \frac{4}{L}\|\nabla\loss_m(\beta^*)\|_\infty$, then for any $\beta\in \Omega$ satisfying
\begin{equation}\label{lem-cone-De2}
\phi(\beta)-\phi(\beta^*)\leq \Delta,
\end{equation}
it holds that
\begin{equation}\label{lem-cone-1}
\begin{aligned}
\|\beta-\beta^*\|_1&\leq 4R_q^{\frac12}\left(\frac{\log n}{m}\right)^{-\frac{q}{4}}\|\beta-\beta^*\|_2+4R_q\left(\frac{\log n}{m}\right)^{\frac12-\frac{q}2}+\frac{2}{L}\min\left(\frac{\Delta}{\lambda},r\right).
\end{aligned}
\end{equation}
Set $\delta:=\beta-\beta^*$. From \eqref{lem-cone-De2}, we obtain that
\begin{equation*}
\loss_m(\beta^*+\delta)+\regu_\lambda(\beta^*+\delta)\leq \loss_m(\beta^*)+\regu_\lambda(\beta^*)+\Delta.
\end{equation*}
Then subtracting $\langle \nabla\loss_m(\beta^*),\delta \rangle$ from both sides of the former inequality, one has that
\begin{equation}\label{lem-cone-2}
\T(\beta^*+\delta)+\regu_\lambda(\beta^*+\delta)-\regu_\lambda(\beta^*)\leq -\langle \nabla\loss_m(\beta^*),\delta \rangle+\Delta.
\end{equation}
We now claim that
\begin{equation}\label{lem-cone-claim}
\regu_\lambda(\beta^*+\delta)-\regu_\lambda(\beta^*)\leq \frac{\lambda L}{2}\|\delta\|_1+\Delta.
\end{equation}
The following argument is divided into two cases. First assume $\|\delta\|_2\leq 3$.
Then it follows from the RSC condition \eqref{alg-RSC1} and \eqref{lem-cone-2} that
\begin{equation*}
\gamma_3\|\delta\|_2^2-\tau_3\frac{\log n}{m}\|\delta\|_1^2+\regu_\lambda(\beta^*+\delta)-\regu_\lambda(\beta^*)
\leq \|\nabla\loss_m(\beta^*)\|_\infty\|\delta\|_1+\Delta\leq \frac{\lambda L}{4}\|\delta\|_1+\Delta.
\end{equation*}
By assumptions \eqref{thm2-lambda} and \eqref{thm2-m}, we obtain that $\lambda L\geq 8\frac{r\tau}{L}\frac{\log n}{m}$. This, together with the facts that $\gamma_3>0$ and that $\|\delta\|_1\leq \|\beta^*\|_1+\|\beta\|_1\leq g(\beta^*)/L+g(\beta)/L\leq 2r/L$, implies \eqref{lem-cone-claim}.
In the case when $\|\delta\|_2>3$, the RSC condition \eqref{alg-RSC2} yields that
\begin{equation*}
\gamma_4\|\delta\|_2-\tau_4\sqrt{\frac{\log n}{m}}\|\delta\|_1+\regu_\lambda(\beta^*+\delta)-\regu_\lambda(\beta^*)
\leq \|\nabla\loss_m(\beta^*)\|_\infty\|\delta\|_1+\Delta
\leq \frac{\lambda L}{4}\|\delta\|_1+\Delta,
\end{equation*}
thus by assumption \eqref{thm2-lambda},
we also arrive at \eqref{lem-cone-claim}.
Let $J$ denote the index set corresponding to the $|S_\eta|$ largest coordinates in absolute value of $\tilde{\delta}$ (recalling the set $S_\eta$ defined in \eqref{S-eta}. It then follows from Lemma \ref{lem-regu} (with $S=S_\eta$) that
\begin{equation}\label{lem-cone-8}
\regu_\lambda(\beta^*)-\regu_\lambda(\beta^*+\delta)\leq
\lambda L(\|\delta_J\|_1-\|\delta_{J^c}\|_1)+2\lambda L\|\beta^*_{S_\eta^c}\|_1.
\end{equation}
Summing up \eqref{lem-cone-8} and \eqref{lem-cone-claim}, one has that $0\leq \frac{3\lambda L}{2}\|\delta_J\|_1-\frac{\lambda L}{2}\|\delta_{J^c}\|_1+2\lambda L\|\beta^*_{S_\eta^c}\|_1+\Delta$,
and consequently, $\|\delta_{J^c}\|_1\leq 3\|\delta_J\|_1+4\|\beta^*_{S_\eta^c}\|_1+\frac{2\Delta}{\lambda L}$. By the definition of the index set $J$ and using the trivial bound $\|\delta\|_1\leq
2r/L$, one has that
\begin{equation}\label{lem-cone-9}
\|\delta\|_1\leq 4\sqrt{|S_\eta}|\|\delta\|_2+4\|\beta^*_{S_\eta^c}\|_1+\frac{2}{L}\min\left(\frac{\Delta}{\lambda},r\right).
\end{equation}
Combining \eqref{s-eta} with \eqref{lem-cone-9} and setting $\eta=\sqrt{\frac{\log n}{m}}$, we arrive at \eqref{lem-cone-1}.
We now verify that \eqref{lem-cone-De2} is held by the vector $\hat{\beta}$ and $\beta^t$, respectively. Since $\hat{\beta}$ is the optimal solution, it holds that $\phi(\hat{\beta})\leq \phi(\beta^*)$, and by assumption \eqref{lem-cone-De1}, it holds that $\phi(\beta^t)\leq \phi(\hat{\beta})+\Delta\leq \phi(\beta^*)+\Delta$. Consequently, it follows from \eqref{lem-cone-1} that
\begin{equation*}
\begin{aligned}
\|\hat{\beta}-\beta^*\|_1
&\leq 4R_q^{\frac12}\left(\frac{\log n}{m}\right)^{-\frac{q}{4}}\|\hat{\beta}-\beta^*\|_2+4R_q\left(\frac{\log n}{m}\right)^{\frac12-\frac{q}2},\quad \mbox{and}\\
\|\beta^t-\beta^*\|_1
&\leq 4R_q^{\frac12}\left(\frac{\log n}{m}\right)^{-\frac{q}{4}}\|\beta^t-\beta^*\|_2+4R_q\left(\frac{\log n}{m}\right)^{\frac12-\frac{q}2} +\frac{2}{L}\min\left(\frac{\Delta}{\lambda},r\right).\\
\end{aligned}
\end{equation*}
By the triangle inequality, we then conclude that
\begin{equation*}
\begin{aligned}
\|\beta^t-\hat{\beta}\|_1
&\leq \|\hat{\beta}-\beta^*\|_1+\|\beta^t-\beta^*\|_1
\leq 4R_q^{\frac12}\left(\frac{\log n}{m}\right)^{-\frac{q}{4}}(\|\hat{\beta}-\beta^*\|_2+\|\beta^t-\beta^*\|_2)+8R_q\left(\frac{\log n}{m}\right)^{\frac12-\frac{q}2}+\frac{2}{L}\min\left(\frac{\Delta}{\lambda},r\right)\\
&\leq 4R_q^{\frac12}\left(\frac{\log n}{m}\right)^{-\frac{q}{4}}\|\beta^t-\hat{\beta}\|_2+\bar{\epsilon}_{\text{stat}}+\epsilon(\Delta).
\end{aligned}
\end{equation*}
The proof is complete.

\subsection*{Proof of Lemma \ref{lem-Tphi-bound}}

By the RSC condition \eqref{alg-RSC1}, Lemma 2 and Lemma \ref{lem-qlambda} (cf. \eqref{lem-qlambda-13}) , one has that
\begin{equation}\label{lem-bound-1}
\bar{\T}(\beta^t,\hat{\beta})\geq (\gamma_3-\frac{\mu_1}{2})\|\beta^t-\hat{\beta}\|_2^2-\tau_3\frac{\log n}{m}\|\beta^t-\hat{\beta}\|_1^2.
\end{equation}
It then follows from Lemma \ref{lem-cone} and assumption \eqref{thm2-m} that
\begin{equation*}
\bar{\T}(\hat{\beta},\beta^t)\geq (\gamma_3-\frac{\mu_1}{2})\|\beta^t-\hat{\beta}\|_2^2-\tau_3\frac{\log n}{m}\|\beta^t-\hat{\beta}\|_1^2\geq -2\tau\frac{\log n}{m}(\bar{\epsilon}_{\text{stat}}+\epsilon(\Delta))^2,
\end{equation*}
which establishes \eqref{lem-bound-T1}.
Furthermore, it follows from the convexity of $\h_\lambda$ that
\begin{equation}\label{lem-bound-2}
\h_\lambda(\beta^t)-\h_\lambda(\hat{\beta})-\langle \nabla\h_\lambda(\hat{\beta}),\beta^t-\hat{\beta} \rangle\geq 0,
\end{equation}
and the first-order optimality condition for $\hat{\beta}$ that
\begin{equation}\label{lem-bound-3}
\langle \nabla\phi(\hat{\beta}),\beta^t-\hat{\beta} \rangle\geq 0.
\end{equation}
Combining \eqref{lem-bound-1}, \eqref{lem-bound-2} and \eqref{lem-bound-3}, one has that
\begin{equation*}
\phi(\beta^t)-\phi(\hat{\beta})\geq (\gamma_3-\frac{\mu_1}{2})\|\beta^t-\hat{\beta}\|_2^2-\tau_3\frac{\log n}{m}\|\beta^t-\hat{\beta}\|_1^2.
\end{equation*}
Then using Lemma \ref{lem-cone} to bound the term $\|\beta^t-\hat{\beta}\|_2^2$ and noting assumption \eqref{thm2-m},
we arrive at \eqref{lem-bound-T2}.
Now we turn to prove \eqref{lem-bound-0}. Define
\begin{equation*}
\phi_t(\beta):=\bar{\loss}_m(\beta^t)+\langle \nabla{\bar{\loss}_m(\beta^t)},\beta-\beta^t \rangle+\frac{v}{2}\|\beta-\beta^t\|_2^2+\h_\lambda(\beta),
\end{equation*}
which is the optimization objective function minimized over the feasible region $\Omega=\{\beta:g(\beta)\leq r\}$ at iteration count $t$. For any $a\in [0,1]$, it is easy to see that the vector $\beta_a=a\hat{\beta}+(1-a)\beta^t$ belongs to $\Omega$ by the convexity of $\Omega$. Since $\beta^{t+1}$ is the optimal solution of the optimization problem \eqref{algo-pga}, we have that
\begin{equation*}
\begin{aligned}
\phi_t(\beta^{t+1}) &\leq \phi_t(\beta_a)=\bar{\loss}_m(\beta^t)+\langle \nabla{\bar{\loss}_m(\beta^t)},\beta_a-\beta^t \rangle +\frac{v}{2}\|\beta_a-\beta^t\|_2^2+\h_\lambda(\beta_a)\\
&\leq \bar{\loss}_m(\beta^t)+\langle \nabla{\bar{\loss}_m(\beta^t)},a\hat{\beta}-a\beta^t \rangle +\frac{va^2}{2}\|\hat{\beta}-\beta^t\|_2^2+a\h_\lambda(\hat{\beta})+(1-a)\h_\lambda(\beta^t),
\end{aligned}
\end{equation*}
where the last inequality is from the convexity of $\h_\lambda$.
Then by \eqref{lem-bound-T1}, one has that
\begin{equation}\label{lem-bound-4}
\begin{aligned}
\phi_t(\beta^{t+1})
&\leq (1-a)\bar{\loss}_m(\beta^t)+a\bar{\loss}_m(\hat{\beta})+2a\tau\frac{\log n}{m}(\epsilon(\Delta)+\bar{\epsilon}_{\text{stat}})^2+\frac{va^2}{2}\|\hat{\beta}-\beta^t\|_2^2+a\h_\lambda(\hat{\beta})+(1-a)\h_\lambda(\beta^t)\\
&\leq \phi(\beta^t)-a(\phi(\beta^t)-\phi(\hat{\beta}))+2\tau\frac{\log n}{m}(\epsilon(\Delta)+\bar{\epsilon}_{\text{stat}})^2+\frac{va^2}{2}\|\hat{\beta}-\beta^t\|_2^2.
\end{aligned}
\end{equation}
Applying the RSM condition \eqref{alg-RSM} on the pair $(\beta^{t+1},\beta^t)$, one has by \eqref{lem-qlambda-11} and the assumption $v\geq 2\gamma_5-\mu_2$ that
\begin{equation*}
\bar{\T}(\beta^{t+1},\beta^t)\leq \left(\gamma_5-\frac{\mu_2}{2}\right)\|\beta^{t+1}-\beta^t\|_2^2+\tau_5\frac{\log n}{m}\|\beta^{t+1}-\beta^t\|_1^2\leq \frac{v}{2}\|\beta^{t+1}-\beta^t\|_2^2+\tau\frac{\log n}{m}\|\beta^{t+1}-\beta^t\|_1^2.
\end{equation*}
Adding $\h_\lambda(\beta^{t+1})$ to both sides of the former inequality yields that
\begin{equation*}
\begin{aligned}
\phi(\beta^{t+1})
&\leq \bar{\loss}_m(\beta^t)+\langle \nabla{\bar{\loss}_m(\beta^t)},\beta^{t+1}-\beta^t \rangle+\h_\lambda(\beta^{t+1})+\frac{v}{2}\|\beta^{t+1}-\beta^t\|_2^2+\tau\frac{\log n}{m}\|\beta^{t+1}-\beta^t\|_1^2\\
&=\phi_t(\beta^{t+1})+\tau\frac{\log n}{m}\|\beta^{t+1}-\beta^t\|_1^2.
\end{aligned}
\end{equation*}
This, together with \eqref{lem-bound-4}, implies that
\begin{equation}\label{lem-bound-5}
\phi(\beta^{t+1})\leq \phi(\beta^t)-a(\phi(\beta^t)-\phi(\hat{\beta}))+\frac{va^2}{2}\|\hat{\beta}-\beta^t\|_2^2+\tau\frac{\log n}{m}\|\beta^{t+1}-\beta^t\|_1^2+2\tau\frac{\log n}{m}(\epsilon(\Delta)+\bar{\epsilon}_{\text{stat}})^2.
\end{equation}
Define $\delta^t:=\beta^t-\hat{\beta}$. Then it follows that
$\|\beta^{t+1}-\beta^t\|_1^2\leq (\|\delta^{t+1}\|_1+\|\delta^t\|_1)^2\leq 2\|\delta^{t+1}\|_1^2+2\|\delta^t\|_1^2$. Combining this inequality with \eqref{lem-bound-5}, one has that
\begin{equation*}
\phi(\beta^{t+1})\leq \phi(\beta^t)-a(\phi(\beta^t)-\phi(\hat{\beta}))+\frac{va^2}{2}\|\hat{\beta}-\beta^t\|_2^2+2\tau\frac{\log n}{m}(\|\delta^{t+1}\|_1^2+\|\delta^t\|_1^2)+2\tau\frac{\log n}{m}(\bar{\epsilon}_{\text{stat}}+\epsilon(\Delta))^2.
\end{equation*}
To simlify the notations, we define $\psi:=\tau\frac{\log n}{m}(\bar{\epsilon}_{\text{stat}}+\epsilon(\Delta))^2$, $\zeta:=\tau R_q\left(\frac{\log n}{m}\right)^{1-\frac{q}2}$ and $\Delta_t:=\phi(\beta^t)-\phi(\hat{\beta})$. Applying Lemma \ref{lem-cone} to bound the term $\|\delta^{t+1}\|_1^2$ and $\|\delta^t\|_1^2$, we obtain that
\begin{equation}\label{lem-bound-6}
\begin{aligned}
\phi(\beta^{t+1})
&\leq \phi(\beta^t)-a(\phi(\beta^t)-\phi(\hat{\beta}))+\frac{va^2}{2}\|\delta^t\|_2^2+64R_q\tau\left(\frac{\log n}{m}\right)^{1-\frac{q}2}(\|\delta^{t+1}\|_2^2+\|\delta^t\|_2^2)+10\psi\\
&= \phi(\beta^t)-a(\phi(\beta^t)-\phi(\hat{\beta}))+\left(\frac{va^2}{2}+64\zeta\right)\|\delta^t\|_2^2+64\zeta\|\delta^{t+1}\|_2^2+10\psi.\\
\end{aligned}
\end{equation}
Subtracting $\phi(\hat{\beta})$ from both sides of \eqref{lem-bound-6}, we have by \eqref{lem-bound-T2} that
\begin{equation*}
\begin{aligned}
\Delta_{t+1}
&\leq (1-a)\Delta_t+\frac{va^2+128\zeta}{\gamma-\mu_1/2}(\Delta_t+2\psi) +\frac{128\zeta}{\gamma-\mu_1/2}(\Delta_{t+1}+2\psi)+10\psi.
\end{aligned}
\end{equation*}
Setting $a=\frac{2\gamma-\mu_1}{4v}\in (0,1)$, one has by the former inequality that
\begin{equation*}
\begin{aligned}
\left(1-\frac{256\zeta}{2\gamma-\mu_1}\right)\Delta_{t+1}&\leq \left(1-\frac{2\gamma-\mu_1}{8v}+\frac{256\zeta}{2\gamma-\mu_1}\right)\Delta_t +2\left(\frac{2\gamma-\mu_1}{8v}+\frac{512\zeta}{2\gamma-\mu_1}+5\right)\psi,
\end{aligned}
\end{equation*}
or equivalently, $\Delta_{t+1}\leq \kappa\Delta_t+\xi(\bar{\epsilon}_{\text{stat}}+\epsilon(\Delta))^2$, where $\kappa$ and $\xi$ were previously defined in \eqref{lem-bound-kappa} and \eqref{lem-bound-xi}, respectively. Finally, we obtain that
\begin{equation*}
\begin{aligned}
\Delta_t &\leq \kappa^{t-T}\Delta_T+\xi(\bar{\epsilon}_{\text{stat}}+\epsilon(\Delta))^2(1+\kappa+\kappa^2+\cdots+\kappa^{t_T+1})\\
&\leq \kappa^{t-T}\Delta_T+\frac{\xi}{1-\kappa}(\bar{\epsilon}_{\text{stat}}+\epsilon(\Delta))^2\leq
\kappa^{t-T}\Delta_T+\frac{2\xi}{1-\kappa}(\bar{\epsilon}_{\text{stat}}^2+\epsilon^2(\Delta)).
\end{aligned}
\end{equation*}
The proof is complete.

\section*{Acknowledgments}

\noindent Chong Li was supported in part by the National Natural Science Foundation of China [grant number 11971429] and Zhejiang Provincial Natural Science Foundation of China [grant number LY18A010004]; Jinhua Wang was supported in part by the National Natural Science Foundation of China [grant number 11771397] and  Zhejiang Provincial Natural Science Foundation of China [grant numbers LY17A010021, LY17A010006]; Jen-Chih Yao was supported in part by the Grant MOST [grant number 108-2115-M-039-005-MY3].



\bibliographystyle{elsarticle-harv}
\bibliography{./nonconvex-arxiv}


%
%
%
\end{document}